\documentclass{amsart}

\usepackage{amsmath,amsfonts,amssymb,amsthm}
\usepackage{graphicx}
\usepackage[all]{xy}
\usepackage{mathrsfs}
\usepackage[colorlinks=true]{hyperref}
\usepackage{enumerate}

\theoremstyle{definition}
\newtheorem{theorem}{Theorem}[subsection]
\newtheorem{theoremz}{Theorem}             

\newtheorem{corollary}[theorem]{Corollary}
\newtheorem{definition}[theorem]{Definition}
\newtheorem{example}[theorem]{Example}

\newtheorem{lemma}[theorem]{Lemma}

\newtheorem{proposition}[theorem]{Proposition}
\newtheorem{remark}[theorem]{Remark}
\numberwithin{equation}{section}

\newcommand{\toto}{\rightrightarrows}
\newcommand{\xto}{\xrightarrow}
\newcommand{\from}{\leftarrow}
\newcommand{\xfrom}{\xleftarrow}

\def\<{\langle}
\def\>{\rangle}
\newcommand{\action}{\curvearrowright}
\newcommand{\raction}{\curvearrowleft}

\newcommand{\R}{\mathbb R}
\newcommand{\N}{\mathbb N}
\newcommand{\Z}{\mathbb Z}

\newcommand{\id}{{\rm id}}

\renewcommand{\d}{{\rm d}}
\newcommand{\G}[1]{G^{(#1)}}
\newcommand{\e}[1]{\eta^{(#1)}}

\newcommand{\al}{\alpha}
\newcommand{\be}{\beta}

\newcommand{\F}{\mathcal F}

\DeclareMathOperator{\supp}{supp}    
\DeclareMathOperator{\Hol}{Hol}         
\DeclareMathOperator{\Mon}{Mon}         
\DeclareMathOperator{\Lie}{Lie}         

\begin{document}

\title{Riemannian metrics on Lie groupoids}

\author{Matias del Hoyo}
\address{Instituto de Matem\'atica Pura e Aplicada, Estrada Dona Castorina 110, Rio de Janeiro, 22460-320, Brazil}
\email{mdelhoyo@impa.br}
\author{Rui Loja Fernandes}
\address{Department of Mathematics, University of Illinois at Urbana-Champaign, 1409 W. †Green Street, Urbana, IL 61801, USA} 
\email{ruiloja@illinois.edu}

\thanks{MdH was partially supported by the ERC Starting Grant No. 279729. RLF was partially supported by NSF grant DMS 1308472 and FCT/Portugal. Both authors acknowledge the support of the \emph{Ci\^encias Sem Fronteiras} grant 401817/2013-0.}

\begin{abstract} 
We introduce a notion of metric on a Lie groupoid, compatible with multiplication, and we study its properties. We show that many families of Lie groupoids admit such metrics, including the important class of proper Lie groupoids. The exponential map of these metrics allow us to establish a Linearization Theorem for Riemannian groupoids, obtaining both a simpler proof and a stronger version of the Weinstein-Zung Linearization Theorem for proper Lie groupoids. This new notion of metric has a simplicial nature which will be explored in future papers of this series.
\end{abstract}

\maketitle

\tableofcontents


\section{Introduction}

This is the first of a series of papers dedicated to the study of Riemannian metrics on Lie groupoids. There have been several attempts to propose a good notion for metrics on Lie groupoids (see, e.g., \cite{gghr,glick,hepworth,ppt}), but to our knowledge, all of them fail to take into account the groupoid multiplication.
The compatibility with the multiplication may not be required in certain particular cases, such as \'etale groupoids, describing orbifolds. But for general Lie groupoids, describing more general singular spaces, taking into account the groupoid multiplication becomes crucial. 
In this paper we introduce a condition of compatibility between the metric and the groupoid multiplication, which is automatically satisfied for \'etale groupoids, but yields a novel concept for more general Lie groupoids.

\smallskip

The basic idea underlying our approach can be easily described using the simplicial model of the Lie groupoid $G\toto M$, say its {\em nerve} $NG$:
$$\xymatrix@1{ \dots \ar@<0.60pc>[r]\ar@<0.30pc>[r] \ar[r] \ar@<-0.30pc>[r]\ar@<-0.60pc>[r]& \G{n}  \ar@<0.45pc>[r]\ar@<0.15pc>[r]\ar@<-0.15pc>[r]\ar@<-0.45pc>[r] &\cdots  \ar@<0.45pc>[r]\ar@<0.15pc>[r]\ar@<-0.15pc>[r]\ar@<-0.45pc>[r]&\G2 \ar@<0.30pc>[r] \ar[r] \ar@<-0.30pc>[r]& \G1 \ar@<0.2pc>[r]\ar@<-0.2pc>[r] & \G0[\ar[r]& \G0/\G1]}.$$

\smallskip

The manifold $\G{n}$ consists of chains of $n$ composable arrows, or equivalently, commutative $n$-simplices on $G$. This viewpoint clarifies the definition of the face maps $\G{n}\to\G{n-1}$, as well as the existence of an action of the symmetric group $S_{n+1}\action\G{n}$ by permuting the vertices of a simplex.
We define an \textbf{$n$-metric on the Lie groupoid $G\toto M$} to be a Riemannian metric $\e{n}$ on $\G{n}$ invariant under $S_{n+1}$ and for which the fibers of one (and therefore every) face map $\G{n}\to\G{n-1}$ are locally equidistant. For small values of $n$ one finds:



\begin{enumerate}[(i)]
\item If $n=0$ we recover the notion of a \emph{transversely invariant metric} on the space of units $M=\G0$, i.e., a metric transversely invariant under the action of $G$ on its units. Such metrics and their properties have been studied in \cite{ppt}.
\item If $n=1$ we recover the notion of Riemannian groupoid introduced first in \cite{gghr}. This is a metric in the space arrows $\G1=G$ invariant under inversion and for which the source and target fibers are equidistant.
\item If $n=2$ we obtain a new notion of metric on a Lie groupoid, a metric in the space of composable arrows $\G2$ which is $S_3$-invariant and transverse to the  multiplication map. We will stress the advantages of this new concept.
\end{enumerate}

\smallskip

An $n$-metric on $\G{n}$ induces a $k$-metric on each $\G{k}$ for $0\le k<n$, for which all face maps $\G{k+1}\to\G{k}$ are Riemannian submersions. Since the nerve of a Lie groupoid is completely determined once one reaches $\G2$, it is enough for many purposes to consider only 2-metrics. We study 2-metrics in this paper and reserve a discussion on general $n$-metrics for the next papers in this series \cite{fdh1,fdh2}. Still, it maybe worth to mention here the following fact, to be proved in \cite{fdh1}: there is at most one 3-metric inducing a given 2-metric in $G\toto M$, and if such a 3-metric exists, it has a unique extension to an $n$-metric for every $n$.
On the other hand, we will give later examples of groupoids which admit an $n$-metric, but do not admit an $n+1$-metric, for $n=0,1,2$. Also, uniqueness fails in these low degrees, so one can have, e.g., two different 2-metrics on $\G2$ inducing the same metric on $\G1$.

\smallskip

A 2-metric $\e2$ on a Lie groupoid $G\toto M$ can be used to great profit. We will call the pair $(G\toto M,\e2)$ 
a \textbf{Riemannian groupoid}. 
As we will see later, there are plenty of classes of groupoids admitting 2-metrics: Lie groups, \'etale groupoids, transitive Lie groupoids, locally trivial Lie group bundles, etc. More important, our first main result shows that:

\begin{theoremz}
\label{thmz:main:1}
Any Hausdorff proper Lie groupoid $G\toto M$ admits a 2-metric.
\end{theoremz}

In fact, a Hausdorff proper groupoid admits an $n$-metric, for any $n$, but we choose to ignore this for now, since our focus in this paper is on 2-metrics.

Our first major application of the existence of 2-metrics is to the \emph{linearization problem} for Lie groupoids. Given a Lie groupoid $G\toto M$ and an embedded saturated submanifold $S\subset M$ (i.e. $S$ is a collection of orbits of $G$), one has a \emph{linear local model} for $G$ around $S$: one can show that the normal bundle to the restriction $G_S\toto S$ is a Lie groupoid $\nu(G_S)\toto\nu(S)$, which only depends on the first jet of $G_S$. Then the linearization problem asks if there exists a groupoid isomorphism from a neighborhood of $G_S$ in $G$ to a neighborhood of $G_S$ in $\nu(G_S)$. 
Special cases of this problem, where the answer is affirmative, include some classical results in differential geometry, such as:
\begin{itemize}
\item Ehresmann's Theorem for submersions,
\item Local Reeb stability for foliations,
\item The Bochner Linearization Theorem and, more generally, the Tube Theorem for proper Lie group actions.
\end{itemize}
In recent years, this problem has been intensively studied by several authors focusing on the case of proper groupoids \cite{cs,ppt,wein1,wein2,zung}. Using 2-metrics and their exponential maps, we will give a simple, geometric proof of the following much more general result:

\begin{theoremz}
\label{thmz:main:2}
Let $(G\toto M,\e2)$ be a Riemannian Lie groupoid and let $S\subset M$ be a saturated embedded submanifold. Then the exponential map defines a linearization of $G$ around $S$.
\end{theoremz}

This theorem yields all known linearization theorems for Lie groupoids. In particular, from Theorems \ref{thmz:main:1} and \ref{thmz:main:2} one can easily obtain the linearization theorem for proper Lie groupoids, where the groupoid neighborhood can be taken to be a full groupoid neighborhood, and the invariant linearization theorem for source proper Lie groupoids, where the groupoid neighborhood can be taken to be saturated (see Section \ref{subsec:linearization} for details).

Summarizing, our metric approach to the linearization problem has the following advantages over the approaches one can find in the literature: (i) it is valid for any Riemannian groupoid, which include proper groupoids but also many other classes, enlarging the range of application, (ii) it holds around any saturated submanifold $S$ and not only around orbits $O$, (iii) it gives some control over the linearization map, and (iv) it has a simple conceptual proof that enlight the subject considerably.

The fact that the linearization is valid around any saturated submanifold $S$, instead of just an orbit $O$, has a nice interpretation from a \emph{stacky} viewpoint: we will prove in \cite{fdh2} that our notion of metric is Morita invariant, hence it leads to a notion of metrics on smooth stacks, an essential tool in the study of their geometry. From this perspective, our linearization around an orbit $O$ yields normal coordinates around a point in the stack, while linearization around a saturated set $S$ yields an exponential tubular neighborhood around a substack.
We will develop these and other aspects of metrics on Lie groupoids in the forthcoming papers \cite{fdh1,fdh2}.

\medskip

\noindent{\bf  Acknowledgments. }
We thank IST-Lisbon, IMPA and U. Utrecht for hosting us at several stages of this project.
We also thank H. Bursztyn, M. Crainic, I. Marcut, D. Martinez-Torres, H. Posthuma and I. Struchiner for many fruitful discussions.

\section{Transversely invariant metrics}

In this section we discuss metrics which are invariant under groupoid actions which are relevant for our study of metrics on groupoids. Since a groupoid action \emph{is not by diffeomorphisms} of the space where the groupoid acts, this notion has some subtleties that one needs to address. 



\subsection{Background on Riemannian submersions}



Given $E$, $B$ smooth manifolds endowed with metrics $\eta^E$, $\eta^B$, a {\bf Riemannian submersion} $p:E\to B$ is a submersion whose fibers are equidistant. This condition amounts to requiring that for every point $e\in E$ the map $\d_ep:T_eF^\bot\to T_bB$ is an isometry, where $b=p(e)$ and $T_eF^\bot$ denotes the subspace of vectors orthogonal to the fiber $F=p^{-1}(b)$.

Given an inner product $\eta$ on a finite dimensional vector space $V$, we denote by $\eta^*$ the {\bf dual inner product} on the dual space $V^*$, which results from the identification $V\cong V^*$ defined by $v\mapsto \eta(v,\cdot)$. In terms of the dual inner products, the condition for a Riemannian submersion can be rephrase as follows: for all $e\in E$ the map $(\d_ep)^*:T_b^*B\to T_eF^\circ$ is an isometry, where $T_eF^\circ$ denotes the annihilator of the vectors tangent to the fiber.


For a Riemannian submersion $p:E\to B$ the metric on $B$ is completely determined by the metric on $E$. In fact, given any submersion $p:E\to B$  and $\eta^E$ a metric on $E$, let us call $\eta^E$ {\bf $p$-transverse} if for all $e,e'$ belonging to the same fiber $F$ the composition 
$T_eF^\circ \from T_b^*B \to T_{e'}F^\circ$
is an isometry. Of course, this can also be rephrased in terms of the orthogonal spaces to the fibers.

When $\eta^E$ is $p$-transverse we can endow $B$ with a {\bf push-forward metric} $p_*\eta^E$, defined as the unique metric $\eta^B$ which makes the maps
$T_b^*B \to T_{e}F^\circ$ isometries. This gives a smooth well-defined metric on $B$, and hence we can deconstruct the notion of Riemannian submersion as follows:
$$p:E\to B \text{ is Riemannian}
 \quad\iff\quad \begin{cases} \eta^E \text{ is $p$-transverse, and}\\p_*\eta^E=\eta^B.\end{cases}$$
 
Just like any manifold can be made into a Riemannian manifold, {\it every submersion can be made Riemannian}:

\begin{lemma}\label{every.submersion.is.Riemannian}
For any submersion $p:E\to B$ there exists a $p$-transverse metric $\eta^E$ on $E$.
\end{lemma}

\begin{proof}
It is enough to set an Ehresmann connection, declare it orthogonal to the fibers, and construct the metric on the fibers arbitrary and on the connection as a pullback of a fixed one in $B$.
\end{proof}

 


The correspondence $\eta\mapsto \eta^*$ between metrics on $V$ and $V^*$ is not linear, and this makes the two points of view on Riemannian submersions, via the annihilators of the fibers or via the normal spaces to the fibers, somewhat different. In fact, we advocate that it is more advantageous to take the cotangent space point of view in the study of Riemannian submersions, as illustrated by the following proposition:

\begin{proposition}
Let $p:E\to B$  be a submersion and let $\{\eta_1,\dots,\eta_k\}$ be $p$-transverse metrics. Then their {\bf cotangent average} $\eta$ defined by
$$ \eta:=\frac{1}{k}\left(\sum_{i=1}^k \eta_i^*\right)^*, $$
is also $p$-transverse.
\end{proposition}

In contrast, the tangent average is usually not $p$-transverse, as shown by the following simple example:

\begin{example}\label{cotangent.principle}
Let $\eta$ be the canonical metric on $\R^2$ and let $\tilde \eta$ be the metric on $\R^2$ with orthonormal frame $\{\partial_x + f(x,y) \partial_y,\partial_y\}$. If $p:\R^2\to\R$ is the projection in the first factor, both metrics are $p$-transverse, with push-forward the canonical metric. The same also holds for their \emph{cotangent average} $\frac 1 2(\eta^* +\tilde\eta^*)^*$. However, the \emph{tangent average} metric $\frac 1 2(\eta+\tilde\eta)$ is not $p$-trasnverse, for the vector $\partial_x+ \frac{1} 2 f(x,y)\partial_y$ is normal to $\partial_y$, its projection does not depend on $y$, but its norm does.
\end{example}

The philosophy that one must work on the cotangent bundle, rather than on the tangent bundle, will occur frequently in the sequel.  For example, for proper groupoids we will use averaging methods to construct suitable metrics, and these methods will only work when considering the dual metrics on the cotangent bundle.


We now establish a simple fact involving {\bf composition of Riemannian submersions} that will be very useful later.

\begin{lemma}
\label{triangle}
Let $q:\tilde E\to E$ be a surjective submersion, let $\eta$ be a $q$-transverse metric on $\tilde E$, and consider another surjective submersion $p:E\to B$.
$$\xymatrix@1{\tilde E \ar[r]^q & E \ar[r]^p & B}$$
Then $\eta$ is $pq$-transverse if and only if the push-forward metric $q_*\eta$ is $p$-transverse. In that case, we have $p_*(q_*(\eta))=(pq)_*(\eta)$.
\end{lemma}

\begin{proof}
Given $\tilde e\in \tilde E$ denote $e=q(\tilde e)$ and $b=p(e)$.
Let us denote by $F_x$ the fiber corresponding to the map $x$.
We have the following commutative diagram, where the horizontal arrows are linear isomorphisms:
$$
\newdir{ (}{{}*!/-5pt/@^{(}}
\xymatrix@R=10pt{
T_{\tilde e}F_{pq}^\circ \ar@<2pt>@{ (->}[d] & 
T_eF_p^\circ \ar[l]  \ar@<2pt>@{ (->}[d] &
T_b^*B \ar[l]\\
T_{\tilde e}F_{q}^\circ  \ar@<2pt>@{ (->}[d] & 
T_e^*E \ar[l]  & & \\
T_{\tilde e}^*\tilde E}$$
Since $\eta$ is $q$-transverse we can push it forward so as the map $T_{\tilde e}F_q^\circ\from T_e^*E$ is an isometry, and hence also its restriction $T_{\tilde e}F_{qp}^\circ\from T_eF_p^\circ$. It becomes now clear that $T_eF_p^\circ\from T_b^*B$ is an isometry if and only if $T_{\tilde e}F_{qp}^\circ\from T_b^*B$ is, from where the result easily follows.
\end{proof}

\begin{remark}\label{pullback metric}
One final observation concerning Riemannian submersions: Given $f_1:E_1\to B$ and $f_2:E_2\to B$ Riemannian submersions, the fibred product $E_1\times_B E_2$ has a natural {\bf pullback metric} $\eta^{E_1\times_B E_2}$. It is defined by
$$\eta^{E_1\times_B E_2}((v,w),(v',w'))=\eta^{E_1}(v,v')+\eta^{E_2}(w,w')-\eta^B(u,u')$$
where $u=\d f_1(v)=\d f_2(w)$ and $u'=\d f_1(v')=\d f_2(w')$. This metric is smooth and has the property that all maps in the commutative diagram 
$$
\xymatrix{
E_1\times_B E_2\ar[d]\ar[r] & E_2\ar[d]^{f_2}\\
E_1\ar[r]_{f_1} & B}
$$
are Riemannian submersions.
\end{remark}




\subsection{Transversely invariant metrics for groupoid actions}

\label{sub:inv.metrics}


Suppose a Lie groupoid $G\toto M$ acts on a space $E\to M$. What does it mean for a metric in $E$ to be invariant under that action? An arrow $g$ in $G$ determines only a diffeomorphism $E_{s(g)}\to E_{t(g)}$ so the notion of invariant metric is not a straightforward generalization of the case of Lie group actions, which we start by recalling.

Given $\theta:G\action E$ a smooth action of a Lie group $G$ on a manifold $E$, for each $g\in G$ we let $\theta_g:E\to E$, $e\mapsto g\cdot e$ denote the translation by $g$. If $\eta^E$ is a Riemannian metric on $E$ then we say that
\begin{enumerate}[(a)]
\item $\eta^E$ is {\bf $\theta$-invariant} if for each $g\in G$ the map $\theta_g:E\to E$ is an isometry.
\item $\eta^E$ is {\bf transversely $\theta$-invariant} if for each orbit $O$ and each $g\in G$ the map $\bar \theta_g:TE/TO\to TE/TO$ is an isometry, or equivalently, if $(\theta_g)^*:TO^\circ\to TO^\circ$ is an isometry.
\end{enumerate}
Clearly (a) implies (b), which is weaker in general. For instance, any metric on $\R$ is obviously transversely invariant under translations $\R\action\R$, but it may not be invariant.

When moving from groups to groupoids, note that the notion of invariance as in (a) does not make sense. If $G\toto M$ is a Lie groupoid and $G\action E$ is an action, there is not a natural lift to a tangent action $G\action TE$. Nevertheless, we can still define transversely invariant metrics as in (b) by using the so-called normal representation, which we now recall.

\smallskip


Given $G\toto M$ a Lie groupoid and $O\subset M$ an orbit, the {\bf normal representation} $\lambda:G_O\action \nu(O)$ is a linear action of the restriction groupoid $G_O\toto O$ over the normal bundle $\nu(O)\to O$. It encodes the linear infinitesimal data around the orbit and plays a fundamental role in the theory (see eg. \cite{survey}). The action is defined by the composition
$$G_O\times_O \nu(O) \xto\cong \nu(G_O)\xto{\overline{dt}} \nu(O)$$
where the first map is the natural identification between the normal bundle of $G_O\subset G$ and the pullback of $\nu(O)$ through the source map \smash{$G_O\xto s O$}. Unraveling this construction, for each arrow \smash{$y\xfrom g x$} in $G_O$ the map $\lambda_g:\nu_x(O)\to \nu_y(O)$ can be geometrically described as follows: if $\gamma$ is a curve on $M$ whose velocity at 0 represents $v\in \nu_x(O)$, and $\tilde\gamma$ is a curve on $G$ such that $\tilde\gamma(0)=g$ and $s\tilde\gamma=\gamma$, then $\lambda_g(v)\in \nu_y(O)$ is defined by the velocity at 0 of $t\tilde\gamma$.

The {\bf conormal representation} $\lambda^*:G_O\action \nu^*(O)$ is defined by dualizing the previous one. Explicitly, we use the natural identification $\nu^*(O)\cong TO^\circ$ and an arrow \smash{$y\xfrom g x$} in $G_O$, acts on the conormal space at $x$ by
$$\lambda^*_g:T_x O^\circ\to T_{y}O^\circ, \quad \phi\mapsto \phi\circ \lambda_{g^{-1}}.$$ 

\begin{remark}\label{rmk:naturality.normal-representation}
The normal and the conormal representations are natural. More precisely, let $\Phi:G\to G'$ be a Lie groupoid map covering $\phi:M\to M'$.  Given \smash{$y\xfrom g x$} an arrow in the orbit $O\subset M$, write $x'=\phi(x)$, etc. and consider the following commutative diagram:
$$\xymatrix{
T_xO^\circ \ar[r]^{ds^ *} & T_g G_O^\circ & T_yO^\circ \ar[l]_{dt^ *}\\
T_{x'}O'^\circ \ar[r]^{ds'^*} \ar[u]^{\d\phi^*} & T_{g'} G_{O'}^\circ \ar[u]^{\d\Phi^*} & T_{y'}O'^\circ \ar[u]^{\d\phi^*} \ar[l]_{dt'^ *}
}$$
The horizontal arrows are all isomorphisms, the upper composition gives the conormal action corresponding to $g$, and the bottom composition the corresponding to $g'=\Phi(g)$, showing the naturality of the conormal representation. The argument for the normal representation is analog.
\end{remark}

If we have a metric $\eta^E$ on the manifold $E$ and $O\subset E$ is a submanifold, then on the normal bundle $\nu(O)$ we consider the metric coming from the identification $\nu(O)\cong  TO^\perp\subset TE$ and on the conormal bundle $\nu^*(O)=(TO)^\circ\subset T^*E$ we consider the restriction of the metric.

\begin{definition}\label{invariant-metric}
Let $\theta:G\action E$ be a Lie groupoid action. We say that a metric $\eta^E$ on $E$ is {\bf transversely $\theta$-invariant} if the normal representation of the action groupoid $G\times_M E\toto E$ is by isometries, or equivalently, if the conormal representation is by isometries.
\end{definition}


\subsection{Transversely invariant metrics and quotients}


Given $G$ a Lie groupoid, there is a correspondence between free and proper actions $G\action P$ and principal $G$-bundles $P\to B$:  the quotient space $P/G$ of a free proper action naturally inherits a smooth manifold structure, for which the projection $\pi:P\to P/G$ is a submersion (see e.g. \cite{survey}). This extends the better known case of free and proper actions of Lie groups and principal group bundles (see e.g. \cite{sharpe}, Appendix E). The next result is a Riemannian version of this correspondence.

\begin{lemma}
\label{quotient.metrics}
Let $\theta:G\action P$ be a free and proper groupoid action with quotient $\pi_P:P\to P/G$. A metric $\eta^P$ on $P$ is transversely $\theta$-invariant if and only if it is $\pi_P$-transverse. In such a case, $P/G$ inherits a well-defined {\bf quotient metric} $\bar\eta=\pi_*\eta$, for which the projection is a Riemannian submersion. 
\end{lemma}

\begin{proof}
The orbits $O$ of $\theta$ coincide with the fibers of $\pi_P:P\to P/G$, so $\eta^P$ is $\pi_P$-transverse if and only if for all $e,e'\in O$ the maps:
$$(T_e O)^\circ \xto{\cong} (T_{e'}O)^\circ,\quad e,e'\in O,$$
are isometries. But this map coincides with the conormal representation along the unique arrow \smash{$e'\xfrom g e$}. Therefore, being transversely $\theta$-invariant amounts to be the same as being $\pi_P$-transverse. The quotient metric is just the push-forward metric as defined in a previous paragraph.
\end{proof}

\begin{example}
\label{ex:left:right:action}
A Lie groupoid acts on its arrows by left translations. This actions is principal with projection the source map $s:G\to M$. Hence a metric $\eta$ on $G$ is transversely left invariant if and only if it is $s$-transverse. Analogously, a metric is transversely right invariant if and only if it is $t$-transverse.
\end{example}


Many of the metrics that we are going to deal with are quotient metrics for free and proper actions, as in Lemma \ref{quotient.metrics}. Then the following technical result becomes an important tool for constructing examples.

\begin{proposition}
\label{quotient.metrics:maps}
Let $(E,\eta^E)$, $(E',\eta^{E'})$ be Riemannian manifolds endowed with free proper groupoid actions $G\action E$, $G\action E'$, for which the metrics are transversely invariant.
If $p:(E,\eta^E)\to (E',\eta^{E'})$ is an equivariant Riemannian submersion then the induced map between the quotients
$\bar p:(E/G,\bar\eta^E)\to(E'/G,\bar\eta^{E'})$
is a Riemannian submersion as well.
\end{proposition}

\begin{proof}
Call $\pi_E:E\to E/G$ and $\pi_{E'}:E'\to E'/G$ the corresponding projections, which are Riemannian submersions by construction (see Lemma \ref{quotient.metrics}):
$$\xymatrix{ E \ar[d]_{\pi_E} \ar[r]^p & E' \ar[d]^{\pi_{E'}} \\ E/G \ar[r]^{\bar p} & E'/G}$$
Since $\pi_{E'}$ and $p$ are Riemannian submersions then so does their composition $\pi_{E'}p=\bar p\pi_E$, and since $\pi_E$ is a Riemannian submersion we conclude that $\bar p$ also is (cf. Lemma \ref{triangle}).
\end{proof}


\section{Riemannian groupoids}
\label{sec:Riemannian:groupoids}


Given $G\toto M$ a Lie groupoid, we let $\G{n}\subset G^n$ be the embedded submanifold consisting of chains of $n$ composable arrows in $G$. Thus $\G0=M$ are the objects, $\G1=G$ are the arrows, and $\G2=\{(h,g)|s(h)=t(g)\}\subset G\times G$ are the pairs of composable arrows, or equivalently the commutative triangles in $G$. We now discuss appropriate notions of metrics on each of these spaces. We emphasize that we will be mostly interested in metrics in $\G2$, which we consider to be the right notion of metric on a Lie groupoid. Nevertheless, the metrics in $\G1$ and in $\G0$, which have been studied before (see \cite{ppt,gghr}), also play a role in our approach.

\subsection{Metrics on the space of units: 0-metrics}


We restate the definition of \cite{ppt} in our language:

\begin{definition}[\cite{ppt}]
A {\bf 0-metric} on a Lie groupoid $G\toto M$ is a Riemannian metric $\e0$ on the manifold $\G0=M$ which is transversely invariant for the canonical action $G\action M$, $g\cdot s(g)=t(g)$.
\end{definition}

In \cite{ppt} the authors develop some aspects of the theory of 0-metrics. One of their main results states that every proper Lie groupoid admits a 0-metric \cite[Prop. 3.14]{ppt}, that for this metric the orbit foliation is a singular Riemannian foliation \cite[Prop. 6.4]{ppt}, and that the corresponding distance function induces a topology on the orbit space that agrees with the quotient topology \cite[Thm. 6.1]{ppt}. Some of these results will be generalized later. 


For now, let us observe that obviously not every Lie groupoid admits a 0-metric. Next we give a counter-example inspired in a foliation with infinite holonomy.

\begin{example}\label{no.example}
The action $\R\action S^1\times\R$ defined by $\lambda\cdot(z,t)=(e^{2\pi i\lambda}z,e^\lambda t)$), leads to a Lie groupoid $G\toto M$ with a single compact orbit $S^1\times 0$. For $x=(1,0)$ the isotropy group $G_x\cong\Z$ is generated by $g=(1,1,0)$, and its normal representation is given by multiplication by $e$, say $g\cdot[\partial_t]=e[\partial_t]$. Hence, there cannot be a 0-metric on $G\toto M$.
\end{example}

Nevertheless, there are plenty of examples of Lie groupoids endowed with 0-metrics. In some cases the invariance condition becomes vacuous, hence a 0-metric amounts to be the same as a metric on the units.

\begin{example}
The unit groupoid $M\toto M$ of a manifold $M$, a Lie group $G\toto\ast$ (viewed as a groupoid with a single object), and more generally any transitive Lie groupoid $G\toto M$, any metric on $M$ is a 0-metric: there is just one orbit, so the condition is vacuous.

Similarly, for a {\bf Lie group bundle} $G\toto M$, that is, a Lie groupoid on which the source and target maps agree, any metric on $M$ is a 0-metric: the orbits are just the points of $M$ and the normal representations are trivial. 
\end{example}

Let us turn to less trivial examples.

\begin{example}
If $p:M\to N$ is a submersion, then a 0-metric on the corresponding submersion groupoid $M\times_N M\toto M$ is the same thing as a $p$-transverse metric on $M$.
\end{example}


The notion of a 0-metric arose first in the context of \emph{\'etale groupoids}. Recall that an {\bf \'etale groupoid} $G\toto M$ is one on which both $G$ and $M$ have the same dimension, and therefore all the structural maps $s,t,m,u,i$ are \'etale, i.e., local diffeomorphisms, and the orbits are 0-dimensional. In an \'etale Lie groupoid every arrow $g$ induces exactly one (germ of a) bisection. The corresponding pseudogroup of transformations of $M$ is called the {\bf effect} of $G$. The following lemma is immediate.

\begin{lemma}\label{0-metrics.on.etal}
A 0-metric on an \'etale Lie groupoid $G\toto M$ is the same as a metric on $M$ invariant under the effect of $G$, that is, a metric for which the bisections act by isometries.
\end{lemma}

The orbit space of a proper effective \'etale groupoid has an {\em orbifold structure}. Conversely, every (effective) orbifold can be obtained in this way, setting a 1-1 correspondence between Morita classes of such groupoids and isomorphism classes of effective orbifolds (see \cite{mm}). Under this correspondence, it is easy to see that a 0-metric on the groupoid amounts to be the same thing as a metric on the orbifold, as defined for instance in loc. cit. (see also \cite[Prop 5.5]{hepworth}).


\begin{remark}
In some works (see, e.g., \cite{glick}) a {\em Riemannian groupoid} $G\toto M$ is defined as a Lie groupoid endowed with a metric $\eta$ on the units for which any bisection acts by isometries. For \'etale groupoids this is just the notion of a 0-metric, but for general groupoids this is not a very useful definition: if a Lie groupoid admits such a metric then all its orbits must be 0-dimensional. 
In fact, if $O\subset M$ is an orbit with positive dimension and \smash{$y\xfrom g x$} is an arrow on $O$, then for arbitrary non-zero vectors $w\in T_yO$ and $v\in T_xO$ we can always find $u\in T_gG$ with $\d t(u)=w$ and  $\d s(u)=v$ (see \cite[Prop. 3.5.1]{survey}). A subspace $S\subset T_gG$ complementing both $\ker_g \d s$ and $\ker_g \d t$ and containing $u$ then leads to a bisection relating $w$ and $v$, from where they should have the same norm, contradicting that they are arbitrary.
\end{remark}


\subsection{Metrics on the space of arrows: 1-metrics}


We recall now the definition of a metric on a groupoid that one can find in \cite{gghr}, which we reformulate in our language:

\begin{definition}[\cite{gghr}]
A {\bf 1-metric} on $G\toto M$ is a metric $\e1$ on the manifold $\G1=G$ which is transversely left invariant (see Example \ref{ex:left:right:action}) and for which the inversion $i$ is an isometry.
\end{definition}

As we have seen, transversely left invariant is the same as $s$-transverse, and assuming that the inversion $i:G\to G$ is an isometry, this is equivalent to $t$-transverse, by Lemma \ref{triangle}.
This gives many equivalent formulations for previous definition. Note that if a metric $\eta$ on $G$ is both $s$-transverse and $t$-transverse, it may not be invariant under inversion, but it leads to a groupoid metric by considering the cotangent average $\e1=(\frac 1 2(\eta^* + i^*\eta^*))^*$.


Given $G\toto M$ a Lie groupoid and $\e1$ a 1-metric, since the inversion is an isometry, by Lemma \ref{triangle} the push-forward metrics $s_*\e1$ and $t_*\e1$ on $M$ agree. We denote this common induced metric on $M$ by $\e0$. Note that $\e0$ differs in general from the restriction of $\e1$ along the unit map, as is observed in \cite{gghr} and we will see in many examples. The proof of the first two items in the following proposition can also be found in \cite{gghr}:

\begin{proposition}
\label{metric.on.the.units}
Given $G\toto M$ a Lie groupoid, $\e1$ a 1-metric and $\e0$ the induced metric on $M$. Then:
\begin{enumerate}[(i)]
\item The source and target maps are Riemannian submersions, and the units $u(M)\subset G$ form a totally geodesic submanifold;
\item $\e0$ makes the orbit foliation $F_M$ of $M$ into a singular Riemannian foliation;
\item $\e0$ is transversely invariant for the action $G\action M$, namely it is a 0-metric.
\end{enumerate}
\end{proposition}

\begin{proof}
The units form a totally geodesic submanifold because they form the fixed point set of the inversion map, which is an isometry. 
That the source and target are Riemannian submersions is immediate. It implies in particular that the foliations on $G$ given by the source and target fibers are Riemannian, namely a geodesic which is orthogonal to some leaf remains orthogonal to every leaf it meets. Now, since for every $g\in G_O$ we have $T_g G_O=\ker \d_g s + \ker \d_g t$, it follows that the foliation $F_G$ on $G$ given by the orbits of the groupoid is also Riemannian. This easily implies that the foliation $F_M=s_*(F_G)$ is also Riemannian, for we can locally lift geodesics along $s$.

Regarding the last statement, given \smash{$y\xfrom g x$} an arrow on $G$, its action under the normal representation consists of the composition $\overline{\d t} \circ\overline{\d s}^{-1}$.
$$N_xO \xfrom{\overline{ds}} N_gG_O \xto{\overline{dt}} N_yO$$
Since the source and target map are Riemannian submersions, and since
$T_gG_O=ds^{-1}(T_xO)=dt^{-1}(T_yO)$, the two maps above are isometries and we are done.
\end{proof}


If $\e1$ is a 1-metric inducing $\e0$ we will say that $\e1$ is an {\bf extension} of $\e0$. Let us start by giving an example of a 0-metric than cannot be extended to a 1-metric.

\begin{example}
Consider the vector field $X=x^2\partial_x$ in $\R$, whose integral curve with initial condition $x\neq 0$ is given by $\gamma_x(\epsilon)=\frac{x}{1-x\epsilon}$.
Denote by $G\toto\R$ the corresponding {\em flow groupoid}: the arrows in $G$ are pairs $(\epsilon,x)$ for which $\gamma_x(\epsilon)$ is defined. For this groupoid, we have $s(\epsilon,x)=x$ and $t(\epsilon,x)=\gamma_x(\epsilon)$, while multiplication is defined in the obvious way.

Now any metric $\e0$ on $\R$ is a 0-metric: if $x\neq 0$ then the normal space to the orbit is trivial, so there is no requirement on $\e0$. If $x=0$, the normal space to the orbit is $\R$, but the normal representation of $G$ on $\R$ is trivial because the linear part of the vector field $X$ is zero, so again there is no condition on the metric. On the other hand, there are no 1-metrics on $G$: this can be seen using either (i) or (ii) in Proposition \ref{metric.on.the.units}: the orbit foliation $F_M$ of $M$ has orbits $\R_{<0}$, $\{0\}$ and $\R_{>0}$, so cannot be a Riemannian foliation; also, there is no metric $\e1$ in $G$ for which \emph{both} the s-fibres and the t-fibres are Riemannian foliations.
\end{example}

On the other hand, there are cases on which the extension is unique, e.g. \'etale groupoids, and still other cases on which the same 0-metric admits many extensions.

\begin{example}\label{pair.groupoid}
Given $\e0$ any metric on the manifold $M$, then $\e1:=\e0\oplus\e0$ yields a 1-metric on the pair groupoid $M\times M\toto M$.  However, not every extension of $\e0$ arises in this way. Consider, for instance, $M=\R$, $a\neq 0$, then we have extensions $\e1$ of the canonical metric $\e0$ on $\R$ of the form:
$$(\e1)^*=\begin{bmatrix}1 & a\\ a & 1\end{bmatrix}.$$
relative to the canonical basis for $\R^2$.
\end{example}

There are classes of groupoids where a 0-metric can always be extended to a 1-metric (although maybe not in unique way, like in the previous example). Recall that a {\bf foliation groupoid} $G\toto M$ is a Lie groupoid for which every isotropy group is discrete. Examples of foliation groupoids include \'etale groupoids and submersion groupoids. Alternative characterizations of foliations groupoids are: (i) a Lie groupoid which is Morita equivalent to an \'etale groupoid, or (ii) a Lie groupoid whose Lie algebroid has injective anchor map (see \cite{cm}).

\begin{proposition}\label{prop:foliation.groupoids}
If $G\toto M$ is a foliation groupoid and $\e0$ is a 0-metric, then there exists a 1-metric extending it.
\end{proposition}

\begin{proof}
Call $q$ the codimension of an orbit on $M$. Given \smash{$y\xfrom g x$} on $G$, we can decompose $T_xM$ and $T_yM$ as orthogonal sums between the tangent spaces to the orbit $O$ and their complements.
The differential of the anchor map
$$\d_g (s,t):T_gG\to T_yM\times T_x M= T_yO\oplus N_yO \oplus T_xO\oplus N_xO$$
is injective, hence we can identify $T_gG$ with its image, which are the vectors $(v_1,v_2,v_3,v_4)$ satisfying $\rho_g [v_4]=[v_2]$, where $\rho_g$ denotes the normal representation. By hypothesis, the norm of $v_2$ equals that of $v_4$. We define (the square of) the norm of $(v_1,v_2,v_3,v_4)$ as that of $(v_1,v_2,v_3,0)$. In other words, we preserve the metric along the orbit and correct it on the normal direction. It is immediate that  the metric $\e1$ defined this way is both $s$-fibred and $t$-fibred, and that is sitting over $\eta$. Since $(s,t)\circ i = \tau \circ (s,t)$ it easily follows that our metric is invariant for the inversion, which completes the proof.
\end{proof}

Note that the construction in this proof does not depend on any choices, so it leads to a completely determined metric $\e1$. Still, the example of the pair groupoid shows that there may be other extensions.


\begin{example}
Given a regular foliation $\F$ on a manifold $M$, we have its {\bf monodromy groupoid} $\Mon(\F)\toto M$ and its {\bf holonomy groupoid} $\Hol(F)\toto M$. The first consists of homotopy classes of paths within a leaf, and the second is a quotient of the previous one, where two paths are identified if they induce the same holonomy. These examples of foliation groupoids may fail to be Hausdorff.

If $M$ is endowed with a metric for which the leaves stay equidistant, namely every geodesic orthogonal to a leaf remains orthogonal to every leaf it meets, then $\F$ is a {\bf Riemannian foliation}. For a Riemannian foliation holonomy and linear holonomy coincide, so the holonomy groupoid acts faithfully by isometries on the normal space to the foliation, and hence it is always Hausdorff.

A 0-metric on $\Hol(F)\toto M$ is the same as a metric on $M$ making the foliation $\F$ Riemannian. Proposition \ref{prop:foliation.groupoids} recovers the main result of \cite{gghr}, namely that the holonomy groupoid of a regular Riemannian foliation has a 1-metric extending the original one.
\end{example}


\subsection{Metrics on the space of composable arrows: 2-metrics}
\label{sec:2:metrics}


The notions of 0-metrics and 1-metrics, discussed in the previous paragraphs, do not take into account the groupoid multiplication $m:\G2\to G$. Hence, they ignore a fundamental ingredient of the concept of a Lie groupoid. This is fixed by considering metrics on the space of composable arrows $\G2$. In order to do that we need first to discuss certain symmetries of the manifold $\G2$, which are formal consequences of the combinatorics underlying its structure.

Each point in $\G2$ represents a pair of composable arrows, or equivalently, a commutative triangle of arrows. Hence we have a natural action of the symmetric group $S_3\action\G2$, by permuting the objects of a given triangle.
For example, in the diagram below, the permutation $(xy)(z)$ sends the point $(h,g)\in\G2$ to $(hg,g^{-1})$:
$$
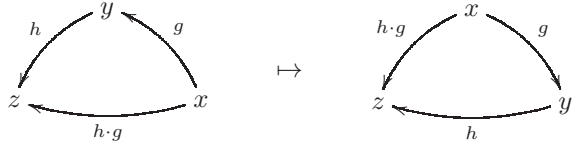
\begin{matrix}
\xymatrix{ &y \ar@/_/[dl]_{h} & \\ z& & x \ar@/_/[ul]_{g} \ar@/^/[ll]^{h\cdot g}}
\end{matrix}
\qquad\mapsto \qquad
\begin{matrix}
\xymatrix{ &x \ar@/_/[dl]_{h\cdot g} \ar@/^/[dr]^{g}& \\ z& & y \ar@/^/[ll]^{h}}
\end{matrix}
$$
We leave it to the reader to find the other transformations for this action in terms of the pair notation in $\G2$.

\begin{remark}
A precise formulation can be given as follows. Denote by $[2]$ the pair groupoid on a set with three objects $\{0,1,2\}$. Then $S_3$ naturally identifies with the groupoid automorphisms $[2]\to[2]$, and $\G2$ with the set of functors $[2]\to G$. Then the action $S_3\action\G2$ is just given by pre-composition.
\end{remark}

There are also three commuting left groupoid actions $\theta^1,\theta^2,\theta^3:G\action\G2$, defined as follows.
$$
\theta^1(k)(g,h)=(kg,h)
\qquad
\theta^2(k)(g,h)=(gk^{-1},kh)
\qquad
\theta^3(k)(g,h)=(g,hk^{-1})
$$
They are free and proper, and their orbits agree with the fibers of the maps $\pi_2,m,\pi_1:\G2\to G$, respectively, hence yielding three principal $G$-bundles. Observe that the action $S_3\action \G2$ interchanges these principal actions $\theta^i$.

\begin{definition}
A {\bf 2-metric} on $G\toto M$ is a metric $\e2$ on the manifold $\G2$ which is transversely invariant for the action $\theta^1:G\action \G2$, $k\cdot (g,h)=(kg,h)$, and for which the group $S_3$ acts by isometries. The pair $(G\toto M,\e2)$ is called a \textbf{Riemannian groupoid}.
\end{definition}

\begin{remark}
Unfortunately, the term \emph{Riemannian groupoid} has been used in the literature with different meanings, as in \cite{gghr} and \cite{glick}, which we have already discussed.
Other approach that may seem natural is to consider metrics that are multiplicative tensors, in the sense of \cite{bc}, but a simple computation shows that being positive definite implies that the groupoid must be 0-dimensional.
We hope to convince the reader that a 2-metric is the right geometric structure that one should add to a Lie groupoid in order to call it \emph{Riemannian}.
\end{remark}

It should be clear that the definition of a 2-metric can be reformulated in many equivalent ways. For example, we can say:
\begin{itemize}
 \item $\e2$ is transversely invariant for (any of) the actions $\theta^i:G\action \G2$ and the group $S_3$ acts by isometries, or equivalently, 
\item $\e2$ is transverse to (any of) the maps $\pi_2,m,\pi_1:\G2\to G$ and the group $S_3$ acts by isometries.
\end{itemize}


In the next section we will construct many examples of 2-metrics. 
For now, let us relate this notion with those introduced before:

\begin{proposition}
\label{2-metric.induces.1-metric}
Let $G\toto M$ be a Lie groupoid. A 2-metric $\e2$ on $\G2$ induces a 1-metric $\e1$ on $\G1$, and hence also a 0-metric on $\G0$.
\end{proposition}

\begin{proof}
Given a 2-metric on $\G2$, since $\e2$ is invariant under the action of $S_3$, which interchanges the maps $m,\pi_1,\pi_2$, it easily follows that the push-forward metrics $m_*\e2,(\pi_1)_*\e2,(\pi_2)_*\e2$ all agree (see Lemma \ref{triangle}). Denote the resulting metric on the manifold $G$ by $\e1$.

We have $i m = m \phi$ for some $\phi\in S_3$, which implies that the inversion map $i$ preserves the metric $\e1$ (again Lemma \ref{triangle}). It remains to show that $\e1$ is transverse to the source map $s:G\to M$, or equivalently, that $\e2$ is transverse to the map $s\pi_2=sm:\G2\to M$ (once more, Lemma \ref{triangle}).

The principal actions $\theta^1,\theta^2:G\action \G2$ can be combined to give a simultaneous action
$$\theta^{12}:\tilde G\action \G2 \qquad (g',g)\cdot [h_1,h_2,h_3]=[gh_1,g'h_2,h_3]$$
where $\tilde G\toto G$ is the so-called {\it arrow groupoid} of $G\toto M$, whose objects are the arrows of $G$ and whose arrows are the commutative squares. The new action $\theta^{12}$ is again free and proper, and its orbits are exactly the fibers of $s\pi_2$.
We will show that $\e2$ is transversely $\theta^{12}$-invariant (Proposition \ref{quotient.metrics}).

By hypothesis, $\e2$ is both transversely $\theta^1$ and $\theta^2$ invariant. And the action groupoids of $\theta^1$ and $\theta^2$ embed into the action groupoid of $\theta^{12}$, via the inclusions
$$ (g',[h_1,h_2,h_3])\mapsto ((g',\id),[h_1,h_2,h_3]) \qquad (g,[h_1,h_2,h_3])\mapsto ((\id,g),[h_1,h_2,h_3]).$$
Then, we can use the naturality of the conormal representation with respect to these inclusions (cf. \ref{rmk:naturality.normal-representation}), and the fact that every arrow $(g',g)$ in $\tilde G$ factors as $(g',\id)(\id,g)$, to conclude the result.
\end{proof}

Thus, given a 2-metric $\e2$ in a Lie groupoid $G\toto M$, there are induced metrics $\e1$ on $G$ and $\e0$ on $M$ and the following five maps are Riemannian submersions.
$$\xymatrix@1{ \G2 \ar@<0.5pc>[r]^{\pi_1} \ar[r]|m \ar@<-0.5pc>[r]_{\pi_2} & G \ar@<0.25pc>[r]^s \ar@<-0.25pc>[r]_t & M}$$ 
Moreover, the metrics $\e2$ and $\e1$ are preserved by the natural actions of the symmetric groups $S_3\action\G2$ and $S_2\action\G1$. 


Given $G\toto M$ a Lie groupoid and $\e1$ a 1-metric, it is natural to ask if there exists a 2-metric $\e2$ inducing $\e1$. If that is the case we will say that $\e2$ is an {\bf extension} of $\e1$. This {\it extension problem}, which consists on deciding whether there is an extension and if such an extension is unique, admit several different answers, depending on the groupoid. In the next section we will show with examples that one can have any of the following situations:
\begin{enumerate}[(i)]
\item there may exists a unique extension,
\item there may be more than one extensions, or
\item there may be no extension at all.
\end{enumerate}

\section{Examples of Riemannian groupoids}



\subsection{First examples}


For some Lie groupoids $G\toto M$ the extension problem from 1-metrics to 2-metrics has exactly a unique solution. Hence for these Lie groupoids our notion of Riemannian groupoid is essentially the same as the one introduced in \cite{gghr}.

\begin{example}[\'Etale groupoids]
In an \'etale Lie groupoid $G\toto M$ the manifolds $M,G,\G2$ all have the same dimension. If $\e1$ is a 1-metric on it, then it induces a 0-metric $\e0$, which is the same as a metric on $M$ invariant under the effect of $G$ (see Lemma \ref{0-metrics.on.etal}). For such a metric $\e1$ it is not hard to see that the pullback metrics $m^*\e1,\pi_1^*\e1,\pi_2^*\e1$ agree and give the only possible extension $\e2$.
\end{example}

More generally,
as we will show in the forthcoming paper of this series \cite{fdh1},
when working with foliation groupoids there exists exactly one extension of a $k$-metric $\e{k}$ to a $(k+1)$-metric $\e{k+1}$, for any $k\geq 1$.

\smallskip

%
%



Next we construct an example of a Lie group, viewed as a groupoid with a single object, on which there are many possible ways to extend the 1-metric.

\begin{example}
Let $G\toto M$ be the abelian group of vectors in the plane, so $G=\R^2$ and $M=\ast$. The canonical metric $\e1$ on $\R^2$ is a 1-metric for our groupoid. Given a metric $\e2$ on $\G2=\R^2\times\R^2$, let us write the matrix of its dual metric on the canonical basis at a point  $a=(x_1,x_2,y_1,y_2)\in\G2=\R^2\times\R^2$ by
$$\begin{bmatrix}
A & B^t \\
B & C
\end{bmatrix}.$$ 
Using coordinates $(z_1,z_2)$ on $\G1=\R^2$, the cotangent linear maps associated to $\pi_2,m,\pi_1:\G2\to\G1$ are given respectively by
$$
d(z_i)\overset{d_a\pi_2}\mapsto d(y_i) 		\qquad 
d(z_i)\overset{d_a m}\mapsto d(x_i)+d(y_i)	\qquad 
d(z_i)\overset{d_a\pi_1}\mapsto d(x_i)
$$
Thus, if $\e2$ is such that $\pi_2,m,\pi_1$ are Riemannian submersions, then the following equations should hold.
$$C=I \qquad A+B+B^t+C=I \qquad A=I$$
This leaves one degree of freedom, namely
$$B=
\begin{bmatrix}
-1/2 & -\beta \\
\beta & -1/2 
\end{bmatrix}$$ 
for some smooth function $\beta:\R^4\to\R$. In order to have a 2-metric on $\R^2$ we need further to require that $S_3\action\G2$ is an isometric action, which easily translates into the following functional equations on $\beta$:
$$\begin{cases}
\beta(x_1,x_2,y_1,y_2) &= -\beta(-y_1,-y_2,-x_1,-x_2)\\
\beta(x_1,x_2,y_1,y_2) &= \beta(y_1,y_2,-x_1-y_1,-x_2-y_2).
\end{cases}$$
Lastly, different 2-metrics $\e2$ extending our original $\e1$ arise from instance from $\beta\equiv 0$ and $\beta\equiv x_1y_1(x_1+y_1)$. Although we do not consider here 3-metrics, it is easy to see that for most choices of $\beta$, the metric $\e2$ cannot be extended to a 3-metric.
\end{example}



Roughly speaking, a 1-metric on a Lie groupoid leaves some freedom along the isotropies, but it completely describes the Riemannian structure on the longitudinal and the transversal directions. This phenomenon can be made more precise in the context of stacks, namely any two different extensions of a 1-metric induce the same metric on the quotient stack. We will pursue this point of view in part III of this series of papers \cite{fdh2}.

\smallskip

We now give an example of a Lie groupoid with a 1-metric which is not extendable to a 2-metric. 

%



\begin{example}
Let $G\toto M$ be a Lie group bundle, i.e., a groupoid where the source and target maps agree. We can think of such a groupoid as a smooth family of Lie groups, parametrized by the base manifold $M$. The orbits are just points and the normal representations are trivial. A 1-metric on a Lie group bundle $G\toto M$ is the same thing as a metric on $G$ transverse to the projection $s=t$ and invariant under the inversion map, so it is easy to construct such 1-metrics on any Lie group bundle (see Lemma \ref{every.submersion.is.Riemannian} and Proposition \ref{cotangent.principle}). However, such a 1-metric may not be extendable, and moreover, there may not exists a 2-metric at all.

For a concrete example, let $(G\toto M)=(\R^3\toto\R)$ be the Lie group bundle settled by
$$s(\lambda,x,\epsilon)=t(\lambda,x,\epsilon)=\epsilon
\qquad
m((\lambda,x,\epsilon),(\lambda',x',\epsilon))=(\lambda+\lambda',x+e^{\lambda \epsilon}x',\epsilon).$$
This can be seen as a 1-parameter family of Lie group structures $\R^2_\epsilon=(\R^2,\cdot_\epsilon)$, where $(\lambda,x)\cdot_\epsilon(\lambda',x')=(\lambda+\lambda',x+e^{\lambda \epsilon}x')$.
In other words, the group $\R^2_\epsilon$ is the semi-direct product $\R\ltimes\R$ under the action $\phi:\R\action\R$, $\phi_\lambda(x)=e^{\lambda \epsilon}x$.

As any Lie group bundle, $G$ admits a 1-metric, but it turns out that $G$ cannot be endowed with a 2-metric. We give now a direct proof of this, and we will provide a more conceptual explanation later. 

Assume that $\e2$ is a 2-metric, and endow $G,M$ with the induced metrics. Write $\pi:G\to M$ and $\pi':\G2\to M$ for the time projections, which are Riemannian submersions. For each $(\lambda,x)\in G_0$ let $(a(\lambda,x),b(\lambda,x),1)\in T_{(\lambda,x,0)}G$ be the orthogonal lift of $\partial_\epsilon\in T_0M$ along $\pi$, and let $w\in T_{\lambda,x,\lambda',x',0}\G2$ be the orthogonal lift of $\partial_\epsilon\in T_0M$ along $\pi'$.

From the equations $\pi\pi_1=\pi'$ and $\pi\pi_2=\pi'$ it follows that $$w=(a(\lambda,x),b(\lambda,x),a(\lambda',x'),b(\lambda',x'),1).$$
Now, the differential of the multiplication on the canonical basis at a central point has the form
$$d m_{(\lambda,x,\lambda',x',0)}=\begin{bmatrix}
1 & 0 & 1 & 0 & 0 \\
0 & 1 & 0 & 1 & x'\lambda\\
0 & 0 & 0 & 0 & 1
\end{bmatrix}$$
so from the equation $\pi m=\pi'$ the next conditions on $w$ arise:
$$\begin{cases}
\begin{matrix}
a(\lambda,x)+a(\lambda',x') & = & a(\lambda+\lambda',x+x')\\
b(\lambda,x)+b(\lambda',x')+\lambda x' & = & b(\lambda+\lambda', x+x').
\end{matrix}
\end{cases}$$
These equations should hold for any choice of $\lambda,x,\lambda',x'$, but it is easy to see that there are no such functions $a,b$. Note for instance that the right hand sides are symmetric and the left hand sides are not.
\end{example}

In fact, we will show later in Example \ref{ex:lie.alg.bundle}, as a consequence of the Linearization Theorem, that if a bundle of Lie groups admits a 2-metric, then its associated bundle of Lie algebras is locally trivial (see also Example \ref{ex:local:trivial:bundles} for a partial converse).
In the example above, the Lie algebra of $\R^2_\epsilon$ has generators $a,b$ and bracket $[a,b]_\epsilon=\epsilon b$, so it is a non-trivial deformation of the abelian Lie algebra of dimension 2, hence a 2-metric cannot exist.

%
%
%


\subsection{The gauge trick}


Every Lie group admits a 2-metric, hence it can be regarded as a Riemannian Lie groupoid. However, since 2-metrics are concerned with the transverse geometry of a Lie groupoid, this is not surprising. There is however one thing to be learned: the same recipe that allow us one to construct 2-metrics on Lie groups also works in several other situations. 


\begin{example}[Lie groups]
\label{ex:Lie:group}
Let $G$ be a Lie group, and let $\eta$ be a right invariant metric on $G$. If $\<,\>$ denontes the value of $\eta$ at the origin, then for any $g\in G$ and any $\phi,\psi\in T_g^*G$ we have
$$\eta^*_g(\alpha,\beta)=\<R_g\alpha,R_g\beta\> \qquad \alpha,\beta\in T_g^*G$$
The product metric $\eta\times\eta\times\eta$ on $G\times G\times G$ is a 2-metric on the pair groupod $G\times G\toto G$, and moreover, it is transversely invariant under the action of $G$ by right multiplication. Hence, it induces a 2-metric on the quotient groupoid, which is actually isomorphic to the group $G$ (cf. Lemma \ref{quotient.metrics} and Proposition \ref{quotient.metrics:maps}).

Note that the induced metric $\e1$ does not agree in general with the original metric $\eta$. A tedious computation shows that the resulting metrics $\e2$ and $\e1$ are given by:
\begin{align*}
(\e2)^*((\al_1,\al_2),(\be_1,\be_2))_{(g,h)} =&
\langle R_{g}^*\al_1, R_{g}^*\be_1\rangle + \langle R_{h}^*\al_2, R_{h}^*\be_2\rangle -\langle L_{g}^*\al_1, R_{h}^*\be_2\rangle -\\
&-\langle R_{h}^*\al_2, L_{g}^*\be_1\rangle +\langle L_{g}^*\al_1, L_{g}^*\be_1\rangle + \langle L_{h}^*\al_2, L_{h}^*\be_2\rangle. \\
(\e1)^*(\al,\be)_{g}=&
\langle R_{g}^*\al, R_{g}^*\be\rangle + \langle L_{g}^*\al, L_{g}^*\be\rangle. 
\end{align*}
\end{example}

\begin{example}[Locally trivial bundles of Lie groups]
\label{ex:local:trivial:bundles}
We can use the existence of 2-metrics on Lie groups to show that any locally trivial bundle of Lie groups admits a 2-metric. First, for a trivial bundle $M\times G\to M$ we can construct a 2-metric on the space of composable arrows $M\times G\times G\to M$: we chose a 2-metric $\eta^G$ in the Lie group $G$ and a metric $\eta^M$ in $M$, and we form the product metric $\eta^M\oplus\eta^G$. Then, for a locally trivial bundle $G\toto M$, we cover $M$ by trivializing open sets $U_\al$ and 2-metrics $\e2_\al$ on the restrictions $G|_{U_\al}$. If $\{\rho_\al\}$ is a partition of unity subordinated to the cover, then we obtain a 2-metric $\e2$ in $G\toto M$ by setting:
$$ \e2:=\left(\sum_{\al}\rho_\al (\e2_\al)^*\right)^*.$$ 
\end{example}


The 2-metric that we have constructed on a Lie group was obtained as the quotient metric of a suitable one in $G^3$. This can be generalized for any Lie groupoid, now considering the manifold $G^{[3]}\subset G^3$ of triples of arrows with the same source, and the map $\pi^{(2)}:G^{[3]}\to\G2$ given by $\pi^{(2)}(h_1,h_2,h_3)= (h_1h_2^{-1},h_2h_3^{-1})$. 
The fibers of $\pi^{(2)}$ coincide with the orbits of the right-multiplication action,
$$
\begin{matrix}
G^{[3]}\raction G\\
(h_1,h_2,h_3)\cdot k = (h_1k,h_2k,h_3k).
\end{matrix}
\qquad \qquad
\begin{matrix}
\xymatrix@R=10pt{\bullet & & \\ \bullet & \bullet \ar[lu]_{h_1} \ar[l]|{h_2} \ar[ld]^{h_3} & \bullet \ar[l]^k \\ \bullet }\end{matrix}
$$
and this action is free and proper, hence defining a principal $G$-bundle. 
The general estrategy will be to define a nice enough metric on $G^{[3]}$ in a way such that it can be pushed forward along $\pi^{(2)}$, and that the resulting metric is a 2-metric.

Notice that the group $S_3$ acts on the manifold $G^{[3]}$ by permuting its coordinates, and there are three left groupoid actions $G\action G^{[3]}$, each consisting in left multiplication on a given coordinate.
$$
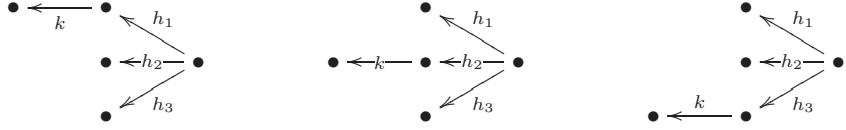
\begin{matrix}
\xymatrix@R=10pt{\bullet & \bullet \ar[l]^k & \\ & \bullet & \bullet \ar[lu]_{h_1} \ar[l]|{h_2} \ar[ld]^{h_3}  \\ & \bullet }\end{matrix}
\qquad\qquad
\begin{matrix}
\xymatrix@R=10pt{  & \bullet & \\\bullet & \bullet \ar[l]|k & \bullet \ar[lu]_{h_1} \ar[l]|{h_2} \ar[ld]^{h_3}  \\ & \bullet }\end{matrix}
\qquad\qquad
\begin{matrix}
\xymatrix@R=10pt{  & \bullet &  \\ & \bullet & \bullet \ar[lu]_{h_1} \ar[l]|{h_2} \ar[ld]^{h_3}  \\\bullet & \bullet \ar[l]_k &}\end{matrix}
$$
These four actions commute with the above right action, and cover the actions $S_3\action\G2$, $G\action\G2$ used when defining 2-metrics.

\begin{remark}
The map $\pi^{(2)}:G^{[3]}\to\G2$ plays a key role for the groupoid. In fact, one can generalize this construction to any $k\geq 0$ in an obvious way, obtaining principal $G$-bundles $\pi^{(k)}:G^{[k+1]}\to\G{k}$, which are $S_{k}$-invariant. They together form the simplicial model for the universal principal $G$-bundle $EG\to BG$. We will come back to this when studying $n$-metrics on Lie groupoids in \cite{fdh1}.
\end{remark}


Emmulating what we have done for Lie groups, we can construct 2-metrics for more general Lie groupoids. The following is a first generalization.

\begin{proposition}\label{gauge.trick.1}
Let $G$ be a Lie group acting on the right on the manifolds $P,N$ and let $q:P\to N$ be an equivariant submersion. Assume that $G\action P$ is free and $G\action N$ is proper. Then the {\bf gauge groupoid} $(P\times_N P)/G\toto P/G$ admits a 2-metric.
\end{proposition}

\begin{proof}
We repeatedly make use of the fact that on a manifold with a proper action an equivariant metric can be constructed by means of a classic averaging argument.
Thus, by setting a preliminary equivariant metric on $P$, which exists for this action is also proper, we can construct an Ehresmann connection $H$ for $q:P\to N$ which is $G$-invariant: take just the orthogonal to the fibers. This way we have $TP=H\oplus F$.

Now, we will modify the preliminary metric by pullbacking to $H$ an equivariant metric on $N$. The resulting metric $\eta^P$ is both $q$-fibred and $G$-invariant

The induced metric $\e2$ on $P\times_N P\times_N P$ make the submersion groupoid $P\times_N P\toto P$ a Riemannian groupoid, and since $\e2$ is $G$-invariant, it induces a quotient metric which yields the desired structure on the gauge groupoid (cf. Lemma \ref{quotient.metrics} and Proposition \ref{quotient.metrics:maps}).
\end{proof}


Let us derive now some immediate corollaries.

\begin{example}[Transitive groupoids]
Every transitive groupoid can be obtained as the gauge groupoid over some free proper action $G\action P$, or equivalently, over some principal $G$-bundle $P\to M$. Hence, by letting $N=\ast$ in the previous proposition, we recover a recipe to construct 2-metrics on any transitive groupoid.
\end{example}

\begin{example}[Proper actions]\label{ex:proper-actions}
If $G\action M$ is a Lie group acting over a manifold, then it is easy to see that the corresponding action groupoid $G\ltimes M\toto M$ is isomorphic to the gauge construction over the projection $G\times M\to M$, where $G$ acts over the product $G\times M$ by $g\cdot(k,x)=(kg^{-1},gx)$. Thus, by the proposition, we may conclude that every Lie groupoid arising from a proper Lie group action admits a 2-metric. Moreover, we can take as the auxiliary metric $\eta^P$ appearing in \ref{gauge.trick.1} the product of a right invariant metric on $G$ and a $G$-equivariant metric on $M$.
\end{example}

We will refer to the construction behind Proposition \ref{gauge.trick.1} as the {\it gauge trick}.






\subsection{2-metrics on proper groupoids}


Recall that a groupoid $G\toto M$ is proper if the anchor map $\rho:G\to M\times M$, $g\mapsto (t(g),s(g))$ is a proper map. In this section we will prove that every Hausdorff proper groupoid admits a 2-metric, which is a groupoid version of the well-known fact that every manifold admits a metric, or that every orbifold admits a metric. We will construct such a 2-metric by adapting the \emph{gauge trick}, introduced in last section.
The delicate point here is that a groupoid action does not lift to an action on the tangent/cotangent bundle. There is however a tangent/cotangent {\it quasi-action}, which will allow us to apply averaging methods to produce 2-metrics. We have collected the relevant material about quasi-actions, tangent lifts and averaging in the Appendix \ref{appendix:averaging}.

\medskip


Let $G\toto M$ be a Lie groupoid and $\sigma$ a connection on it. We consider an action $G\action E$ and its cotangent lift $(G\ltimes E)\tilde\action T^*E$. In order to simplify the notation we will write just $ge=\theta_g(e)$, $gv=T_\sigma\theta_{(g,e)}(v)$ and so on.

Given a metric $\eta$ on $E$, we can view it as a section of the second symmetric power, say $\eta\in \Gamma(E,S^2(T^*E))$, and analogously $\eta^*\in\Gamma(E,S^2(TE))$.
Then the metric $\eta$ is transversely invariant if and only if for each orbit $O\subset E$ the section $\eta^*|_{TO^\circ\times TO^\circ}\in\Gamma(O,S^2(TE/TO))$ is an invariant section for the corresponding lifted action (cf. Proposition \ref{tan.lift.nor.representation}). 

Now suppose that $G\toto M$ is proper, so we can fix $\mu$ a Haar density on it, and consider its associated averaging operators (cf. Definition \ref{def:aver.oper}).

\begin{definition}
Let $\theta:G\action E$ be a groupoid action, and $\eta$ a metric on $E$. Its {\bf cotangent average} $\tilde\eta\in \Gamma(E,S^2(T^*E))$ is defined by averaging its dual, say
$$(\tilde\eta)^*_e(\alpha,\beta):=I_\theta(\eta^*)_e(\alpha,\beta)=
\int_{G(-,x)} \eta^*_{ge}(g\alpha,g\beta) \mu^x(g),$$
where $I_\theta$ is the operator on $\Gamma(E,S^2(TE))$, $x=q(e)$, \smash{$y\xfrom g x$}, and $\alpha,\beta\in T_e^*E$.
\end{definition}


The following proposition plays a key role in the paper:

\begin{proposition}\label{cotangent.average}
Given $\eta$ a metric on $E$, then:
\begin{enumerate}[(i)]
 \item Its cotangent average $\tilde\eta$ is a transversely $G$-invariant metric on $E$.
 \item If $\eta$ is already transversely $G$-invariant, then $\eta$ and $\tilde\eta$ agree in the directions normal to the orbits.
\end{enumerate}
\end{proposition}
\begin{proof}
It is easy to see that the section $\tilde\eta$ is positive definite, hence a metric. To see that it is transversely $G$-invariant, let $O\subset E$ be an orbit, and consider the following vector bundle map:
$$S^2(TE|_O)\to S^2(TE|_O/TO) \qquad \eta^* \mapsto \eta^*|_{TO^\circ\times TO^\circ}.$$
This is surjective and equivariant, where we endow $S^2(TE|_O)$ with the quasi-action induced by the tangent lift, and $S^2(TE|_O/TO)$ with the conormal representation (cf. Proposition \ref{tan.lift.nor.representation}).
By Proposition \ref{prop:aver.oper}(iv) the metric $\tilde\eta|_{TO^\circ\times TO^\circ}$ agrees with the averaging of the restriction $\eta^*|_{TO^\circ\times TO^\circ}$ with respect to the conormal representation. This yields an invariant section by \ref{prop:aver.oper}(i), showing that $\tilde\eta$ is transversely $G$-invariant.
The last statement readily follows from \ref{prop:aver.oper}(ii).
\end{proof}

Notice that if the action $\theta:G\action E$ is free and proper, hence defining a principal $G$-bundle, then the push-forward metric $q_*\tilde\eta$ of the cotangent average along $q:E\to E/G$ is a well-defined metric in the quotient, and if $\eta$ is already transversely $G$-invariant, then we have $q_*\tilde\eta=q_*\eta$.


Let us show now that the cotangent average behaves well with respect to Riemannian equivariant submersions.

\begin{proposition}
\label{prop:cot.ave.&.submersions}
Let $p:(E,\eta^E)\to (B,\eta^B)$ be a Riemannian submersion which is equivariant for actions $\theta^E:G\action E$ and $\theta^B:G\action B$. If $\tilde\eta^E,\tilde\eta^B$ denote the cotangent averages of the metrics, then $p:(E,\tilde\eta^E)\to(B,\tilde\eta^B)$ is also a Riemannian submersion.
\end{proposition}

\begin{proof}
Given $e\in E$ and $\alpha,\beta\in T_b^*B$, where $b=p(e)$, by using first that $p$ is Riemannian and then that is equivariant, we get the following chain of equalities,
\begin{align*}
(\tilde\eta^{B}_b)^*(\alpha,\beta)&=\int_{G(-,x)} (\eta_{gb}^{B})^*(g\alpha,g\beta)\mu^x(g) \\
&=\int_{G(-,x)} (\eta_{ge}^{E})^*((\d_{ge} p)^*g\alpha,(\d_{ge} p)^*g\beta)\mu^x(g)\\
&=\int_{G(-,x)} (\eta_{ge}^{E})^*(g(\d_e p)^*\alpha,g(\d_e p)^*\beta)\mu^x(g)\\
&=(\tilde\eta_e^{E})^*((\d_e p)^*\alpha,(\d_e p)^*\beta),
\end{align*}
from which we conclude that $p$ is also Riemannian for the averaged metrics. 
\end{proof}


%
%


%

Now we have collected all the preliminaries needed to establish our first fundamental theorem.

\begin{theorem}\label{thm:main:1}
Every Hausdorff proper groupoid $G\toto M$ admits a 2-metric $\e2$.
\end{theorem}

\begin{proof}
Endow the manifold $G$ with a Riemannian structure $\eta^{[1]}$ transverse to the source map $G\to M$. For each $k=1,2,3,\dots$ endow the $k$-fold pullback along the source map $G^{[k]}$ with the corresponding pullback metric $\eta^{[k]}$ (cf. \ref{pullback metric}). Then every face map of the submersion groupoid arising from the source is a Riemannian submersion.
$$\xymatrix@1{
\cdots \,
G^{[3]} \ar@<0.5pc>[r]^{\pi_1} \ar[r]|{m} \ar@<-0.5pc>[r]_{\pi_2} &
G^{[2]} \ar@<0.25pc>[r]^{s} \ar@<-0.25pc>[r]_{t} & 
G^{[1]}}$$ 
Moreover, each of these Riemannian submersions is equivariant for the right action $G^{[k]}\raction G$. Thus, after replacing each metric $\eta^{[k]}$ by its cotangent average $\tilde\eta^{[k]}$ we still have that every face map is a Riemannian submersion (cf. Proposition \ref{prop:cot.ave.&.submersions}).
But now we can push-forward each metric $\eta^{[k+1]}$ through the quotient map $G^{[k+1]}\to\G{k}$ and obtain a (fully extendable) $k$-metric on $G\toto M$. In particular, for $k=2$, we obtain the desired 2-metric.
\end{proof}

Notice that the proof shows that the 2-metric is a {\it simplicial metric}, i.e., it extends to an $n$-metric, for all $n\in\N$. We will come back to this notion of metrics in the forthcoming paper \cite{fdh1}.


\section{Linearization of Riemannian Groupoids}
\label{section.linearization}

\subsection{The linearization problem}
\label{subsec:linearization}


Let $G\toto M$ be a Lie groupoid, and let $S\subset M$ be an embedded saturated submanifold of codimension $q$, i.e.,  $S$ is a submanifold which is a union of orbits of $G$. We denote by $G_S\subset G$ the set of arrows whose source and target belong to $S$:
$$ G_S=s^{-1}(S)=t^{-1}(S). $$
Note that $G_S$ is an embedded submanifold of the same codimension $q$ as $S$. We denote by $G_S^{(k)}\subset G^{(k)}$ the set of $k$-tuples of composable arrows in $G_S$, which is again an embedded submanifold of codimension $q$.


For any saturated submanifold of $G\toto M$, there is a {\bf local linear model} for $G$ around  $S$. It can be defined the groupoid-theoretic 
normal bundle:
$$\xymatrix{ \nu(G_S) \ar@<0.25pc>[r] \ar@<-0.25pc>[r]  \ar[d] & \nu(S) \ar[d] \\
G_S \ar@<0.25pc>[r] \ar@<-0.25pc>[r] & S}$$ 
Its objects and arrows are given by the total spaces of the normal bundles $\nu(S)=T_SM/TS$ and $\nu(G_S)=T_{G_S}G/TG_S$, respectively. The structural maps $s,t,m,u,i$ are induced by the total differential of those of $G\toto M$. Notice that this local linear model depends only on the 
linear infinitesimal data around $S$.

\begin{remark}\label{rmk:core-zero}
The structure maps of the local linear model $\nu(G_S)\toto\nu(S)$ are vector bundle maps and fiberwise isomorphisms. One consequence is that $\nu(G_S)$ identifies with the pullback $G_S \times_S \nu(S)$ along the source map, yielding a representation
$$G_S \times_S \nu(S)\cong \nu(G_S) \xto t \nu(S).$$
When $S$ is an orbit this recovers the normal representation recalled in Section \ref{sub:inv.metrics}.
Another consequence, that will play a key role later, is that the pairs of composable arrows of the local linear model $\nu(G_S)^{(2)}$ canonically identifies with the normal bundle \smash{$\nu(G_S^{(2)})$}, for it can be seen as the top square of a cartesian cube.
\end{remark}


The linearization problem consists on establishing an isomorphism between the local model and the original groupoid on suitable neighborhoods. There are however several possibilities for this choice of neighborhoods, as we now explain. 

\begin{definition}
Let $G\toto M$ be a Lie groupoid and let $S\subset M$ be a saturated submanifold. A {\bf groupoid neighborhood} of $G_S\toto S$ in $G\toto M$ is a pair of open subsets $S\subset U\subset M$ and $G_S\subset \tilde{U} \subset G$ such that $\tilde{U}\toto U$ is a subgroupoid of $G\toto M$. Such a groupoid neighborhood is called {\bf full} if: 
$$\tilde{U}=G_U=s^{-1}(U)\cap t^{-1}(U).$$ 
\end{definition}

We note that for proper Lie groupoids we have:

\begin{lemma}
\label{lem:full:neighborhood}
For a proper Lie groupoid, a groupoid neighborhood always contains a full neighborhood.
\end{lemma}

\begin{proof}
Let $G\toto M$ be a proper groupoid, let $S\subset M$ be a saturated submanifold, and let $\tilde{U}\toto U$ be a groupoid neighborhood of $G_S\toto S$. 

We start by observing that each $x\in S$ has a neighborhood $V_x$ in $M$ such that $G_{V_x}\subset \tilde{U}$. This is because $G_x$ is the fiber over $(x,x)$ of the proper map $\rho=(t,s):G\to M\times M$, and therefore, the open $\tilde{U}$ that contains $G_x$ must also contain an open tube $(t,s)^{-1}(V_x\times V_x)=G_{V_x}$.
Since $S$ is second countable, we can find a countable family $\{V_n\}$ with $G_{V_n}\subset \tilde{U}$ and such that $S\subset \bigcup_{n=1}^\infty V_n$. We can further assume that $\overline{V}_n$ is compact for all $n$.

The naive tentative would be to take the full neighborhood given by $\bigcup_{n=1}^\infty V_n$, but as a matter of fact, there may be arrows from some $V_i$ to some $V_j$ not contained in $\tilde U$.
To solve this we shrink each $V_n$ by defining $V'_n=V_n-C_n$,
where
\begin{align*} 
C_n:&=\pi_1((M\times \cup_{i=1}^{n-1}\overline{V}_i)\cap \rho(G-\tilde{U}))\\
&=\{y\in M: \exists g\in G-\tilde{U}\text{ with }t(g)=y\text{ and }s(g)\in \cup_{i=1}^{n-1}\overline{V}_i\}.
\end{align*}
Since both $\rho:G\to M\times M$ and the projection $\pi_1:M\times \cup_{i=1}^{n-1}\overline{V}_i\to M$ are closed maps, we have that $C_n$ is closed, and hence $V'_n$ is open.

Since $G_S\subset \tilde{U}$, it follows that $S\cap C_n=\emptyset$, then $S\cap V_n=S\cap V'_n$ and the collection $\{V'_n\}$ still covers $S$. We finally define $V:=\bigcup_{n=1}^\infty V'_n$, it is easy to check that $G_V\subset \tilde{U}$, and the lemma follows.
\end{proof}

However, for a general Lie groupoid, a groupoid neighborhood may fail to contain a full neighborhood: for example, for the Lie groupoid associated with the flow of a vector field, if $S=\{x_0\}$ is a non-degenerate sink of the vector field, then small enough groupoid neighborhoods do not contain any full neighborhoods. 

Using the notion of groupoid neighborhood we can formulate the various versions of the linearization problem:

\begin{definition}
Let $G\toto M$ be a Lie groupoid and let $S\subset M$ be a saturated submanifold. Then we say that:
\begin{enumerate}[(a)]
\item $G$ is {\bf weakly linearizable} at $S$ if there are groupoid neighborhoods $\tilde{U}\toto U$ of $G_S\toto S$ in $G\toto M$ and $\tilde{V}\toto V$ of $G_S\to S$ in the local model $\nu(G_S)\toto\nu(S)$, and an isomorphism of Lie groupoids:
$$(\tilde{U}\toto U)\overset\phi\cong (\tilde{V}\toto V)$$
which is the identity on $G_S$.
\item $G$ is {\bf linearizable} at $S$ if both $\tilde{U}$ and $\tilde{V}$ can be chosen to be full neighborhoods, so that  there is an isomorphism of Lie groupoids
$$(G_U\toto U)\overset\phi\cong (\nu(G_S)_V\toto V),$$
which is the identity on $G_S$.
\item $G$ is called {\bf invariantly linearizable} at $S$ if it is linearizable and both $U$ and $V$ can be taken to be saturated.
\end{enumerate}
\end{definition}

The linearization problem has been intensively studied in the last decade in the case of \emph{proper} Lie groupoids. See, e.g., \cite{cs} and references therein for the most updated account of the linearization problem in the proper case.

\begin{example}
Let $\R\times \Z\toto\R$ be the trivial bundle of Lie groups with fiber $\Z$ and let $G\toto\R$ be the subgroupoid with fibers $G_t=\Z$ for $t\neq 0$ and $G_0=\ast$, and let $S=\{0\}$. It is easy to see that $G\toto\R$ is weakly linearizable at $S$ but it is not linearizable at $S$.
\end{example}

\begin{example}
If $G=M\times_N M\toto M$ is the Lie groupoid arising from a submersion $p:M\to N$ and $S\subset M$ is the preimage of any embedded submanifold, then we will see below that $G\toto M$ is always linearizable around $S$.
However, it is invariantly linearizable if and only if the submersion is locally trivial at the points of $S$.
\end{example}

\begin{example}
A Lie groupoid arising from a proper action of a Lie group is invariantly linearizable around an orbit. This is in fact a way to rephrase the Tube Theorem, see e.g. \cite{dk}.
\end{example}





\subsection{Exponential neighborhoods}


Let $(M,\eta)$ be a Riemannian manifold. Denote by $\mathcal E_M\subset TM$ the domain of the {\bf exponential map},  that is, 
the open set consisting of tangent vectors $v\in TM$ for which the corresponding geodesic $\gamma_v(t)$ is defined for all $0\le t \le 1$. 
Then the exponential map
$$\exp:\mathcal E_M\to M \qquad \exp(v)=\gamma_v(1)$$
is smooth, it is the identity over $M$ (viewed as the zero section), and its differential at points of $M$ has the form:
$$ T_{0_x}(TM)\cong T_xM\times T_xM\to TM,\quad (v,w)\mapsto v+w.$$


Let $S\subset (M,\eta)$ be an embedded submanifold. We identify the abstract normal bundle $\nu(S)$ with the orthogonal bundle $(TS)^\bot$.
An open subset $S\subset U\subset \mathcal E_M\cap \nu(S)$ is called an {\bf admissible neighborhood} if the exponential map is injective and \'etale over $U$, hence an open embedding.  Then we call the image $\exp(U)\subset M$ an {\bf exponential neighborhood} of $S$. Of course, this is the standard way tubular neighborhoods are constructed for $S\subset M$ out of $\eta$. When $S$ consist of a single point this construction yields normal coordinates around the point (see e.g. \cite{lee}).


The existence of admissible opens is well-known:

\begin{lemma}
Any submanifold $S$ of a Riemannian manifold $(M,\eta)$ has an admissible open neighborhood $S\subset U\subset \mathcal E_M\cap \nu(S)$.
\end{lemma}



We are interested in the construction of exponential neighborhoods related to a Riemannian submersion $p:(E,\eta^{E})\to(B,\eta^B)$.
Let us denote by $H\subset TE$ the horizontal vector bundle, consisting of vectors orthogonal to the fibers. A curve $\tilde\gamma$ on $E$ is {\bf horizontal} if its tangent vectors $\dot{\tilde{\gamma}}(t)$ belong to $H$. Since $H$ is an example of an Ehresmann connection, one has local lifting of curves, i.e., for any $e\in E$ and any curve $\gamma$ in $B$ with $p(e)=\gamma(0)$ there exists a unique horizontal curve $\tilde\gamma$ such that $\tilde\gamma(0)=e$ and $\gamma(t)=p(\tilde\gamma(t))$.


Since $p:(E,\eta^{E})\to(B,\eta^B)$ is a Riemannian submersion, a horizontal curve is a geodesic if and only if its projection is a geodesic. In particular, if a geodesic on $E$ is normal to a fiber, then it is normal to every fiber it meets.

\begin{proposition}\label{prop:commuting.exponential}
Let $p:(E,\eta^{E})\to(B,\eta^B)$ be a Riemannian submersion. If $S\subset B$ is an embedded submanifold and $\tilde S=p^{-1}(S)$, then for any open subsets  $\tilde S\subset \tilde U\subset \mathcal E_E\cap \nu(\tilde S)$ and $S\subset U\subset\mathcal E_B\cap \nu(S)$ such that $\d p(\tilde U)\subset U$, the following square commutes:
$$\xymatrix{
 \tilde U \ar[r]^{\exp} \ar[d]_{\d p} & E \ar[d]^p\\
 U  \ar[r]^{\exp} & B
}$$
Moreover, if $U$ is admissible then $\tilde U$ is also admissible.
\end{proposition}

\begin{proof}
Any vector normal to $\tilde S$ is, in particular, normal to the corresponding fiber. Hence, it gives rise to a horizontal geodesic whose projection is also a geodesic. Therefore, the diagram in the statement of the proposition commutes.

Assume now that $U$ is admissible. To show that $\tilde U$ is also admissible we need to show that $\exp:\tilde U\to E$ is both injective and \'etale:
\begin{itemize}
\item To prove that $\exp:\tilde U\to E$ is injective, lets assume that $\exp(v)=\exp(v')$ and denote by $\gamma$ and $\gamma'$ the geodesics arising from $v$ and $v'$, respectively. Their projections $p(\gamma)$ and $p(\gamma')$ are geodesics in $U$ arising from $\d p(v))$ and $\d p(v')$. Since $\exp$ is injective over $U$ we conclude that $\d p(v)=\d p(v')$. Thus we see that $\gamma$ and $\gamma'$ are two horizontal lifts of the same curve which end at the same point. By the uniqueness of lifting we conclude that they are the same curve. This establishes injectivity.

\item To prove that $\exp:\tilde U\to E$ is \'etale, let $v\in\tilde U$ and set $e=\exp(v)$. The above commutative diagram gives the following map of short exact sequences:
$$\xymatrix{
0 \ar[r] & \ker\d_v (\d p)|_{\tilde U} \ar[r] \ar[d] & T_{v}\tilde U \ar[d]_{\d_v\exp} \ar[r] & T_{\d p(v)}U \ar[d] \ar[r] & 0\\
0 \ar[r] & \ker\d_e p \ar[r] & T_e E \ar[r] & T_{p(e)} B \ar[r] & 0
}$$ 
The last vertical arrow is an isomorphism because $U$ is admissible.
The first vertical arrow identifies the tangent spaces to the fibers of $p$ and $dp$.
It follows that the middle arrow is also an isomorphism, so $\exp:\tilde U\to E$ is a local diffeomorphism, as claimed.
\end{itemize}
\end{proof}


\begin{remark}\label{Ehresmann.Theorem}
The previous proposition establishes the existence of exponential neighborhoods adapted to a Riemannian submersion. When $S=\{y\}$ consists of a single point, it follows that a submersion looks like a projection around a fiber, since in this case the normal bundle $\nu(\tilde S)\cong p^*T_yB\cong p^{-1}(y)\times \R^n$ is trivial. This statement can be thought of as a structure theorem for a submersion. Ehresmann's Theorem, asserting that a proper submersion is locally trivial, can be easily obtained from this statement.
\end{remark}



\subsection{The Linearization Theorem for Riemannian groupoids}

We have now everything in place to state and prove one of our main results.

\begin{theorem}
\label{thm:main:2}
Let $G\toto M$ be a Lie groupoid endowed with a 2-metric $\eta^{(2)}$, and let $S\subset M$ be a saturated embedded submanifold. Then the exponential map defines a weak linearization of $G$ around $S$.
\end{theorem}

\begin{proof}
Let $S\subset V\subset \nu(S)$ be an admissible neighborhood for $\e0$, and let
$$\tilde V = (\d s)^{-1}(V)\cap (\d t)^{-1}(V) \cap \mathcal E_G \cap \nu(G_S)$$
be the arrows of $\nu(G_S)\toto \nu(S)$, with source and target in $V$, belonging to the domain of the exponential map $\exp^{(1)}$ of the metric $\e1$. 

Let us show that $\tilde V\toto V$ is a groupoid neighborhood of $G_S\toto S$ in $\nu(G_S)\toto\nu(S)$. The compatibility with the structural maps $s,t,i,u$ follows from Proposition \ref{metric.on.the.units}. To see that $\tilde V$ is closed under multiplication, let $({\tilde w},{\tilde v})\in\tilde V\times_V\tilde V\subset \nu(G_S)^{(2)}$, and identify $\nu(G_S)^{(2)}\cong \nu(G_S^{(2)})$ in the canonical way (cf. Remark \ref{rmk:core-zero}).

We claim that the geodesic $\gamma_{(\tilde w,\tilde v)}$ in $\G2$ with initial condition $({\tilde w},{\tilde v})$ is  the curve $(\gamma_{\tilde w}(t),\gamma_{\tilde v}(t))$. In fact, note that $\gamma_{(\tilde w,\tilde v)}(t)$ is perpendicular to the fibers of $\pi_1$ and $\pi_2$ at $t=0$. Since $\pi_1,\pi_2:\G2\to G$ are both Riemannian submersions, $\gamma_{(\tilde w,\tilde v)}(t)$ stays perpendicular to those fibers, and we conclude that $\pi_1(\gamma_{(\tilde w,\tilde v)}(t))$ and $\pi_2(\gamma_{(\tilde w,\tilde v)}(t))$ are both geodesics in $G$ with initial conditions ${\tilde w}$ and ${\tilde v}$. Hence, we must have:
$$\gamma_{(\tilde w,\tilde v)}(t)=(\pi_1(\gamma_{(\tilde w,\tilde v)}(t)),\pi_2(\gamma_{(\tilde w,\tilde v)}(t))=(\gamma_{\tilde w}(t),\gamma_{\tilde v}(t)),$$
as claimed. We have proven that $\tilde V\times_V \tilde V$ is included in the domain of $\exp^{(2)}$.

Now, since $m$ is also Riemannian and $({\tilde w},{\tilde v})$ is perpendicular to an $m$-fiber, it follows also that $m(\gamma_{\tilde w}(t),\gamma_{\tilde v}(t))$ is a geodesic in $G$: it is actually the geodesic with initial condition $\d m({\tilde w},{\tilde v})$. We conclude that $\d m({\tilde w},{\tilde v})$ belongs to the domain of $\exp^{(1)}$, hence in $\tilde V$, proving finally that $\tilde V\toto V$ is a groupoid neighborhood. 

Proposition \ref{prop:commuting.exponential} applied to the Riemannian submersions $s,t,m,\pi_1,\pi_2$
shows that we have a commutative diagram:
$$\xymatrix@R=15pt{
\tilde V\times_V \tilde V \ar[rr]^{\exp^{(2)}}\ar@<0.5pc>[d] \ar[d] \ar@<-0.5pc>[d]& & \G2 \ar@<0.5pc>[d] \ar[d] \ar@<-0.5pc>[d] \\
\tilde V \ar[rr]^{\exp^{(1)}} \ar@<0.25pc>[d] \ar@<-0.25pc>[d] & & G\ar@<0.25pc>[d] \ar@<-0.25pc>[d]\\
V \ar[rr]^{\exp^{(0)}}  & &M}$$
where $V$, $\tilde V$ and $\tilde V\times_V \tilde V$ are all admissible. 
We conclude that the exponential maps of $\e1$ and $\e0$ give the desired weak linearization:
$$(\nu(G_S)\toto\nu(S))\supset(\tilde V\toto V)\overset{\exp}\cong (\exp(\tilde V)\toto \exp(V))\subset (G\toto M).$$
\end{proof}

\begin{remark}
The main step in the proof of Theorem \ref{thm:main:2} was showing the following property of the metric $\e1$: if $\tilde w,\tilde v\in \nu(G_S)$ are such that $\d s(\tilde w)=\d t(\tilde v)$, then 
$(\gamma_{\tilde w}(t),\gamma_{\tilde v}(t))$ belongs to $\G2$ for all $t$ and $m(\gamma_{\tilde w}(t),\gamma_{\tilde v}(t))$ is the geodesic in $G$ with initial condition $\d m(\tilde w,\tilde v)$. If one has a metric on $G$ with this property, then the proof shows that one can linearize $G$ around $S$. Notice that this condition on the metric involves multiplication and does not require \emph{a priori} any metric on $\G2$. One is tempted to call a groupoid with a metric satisfying this property for any invariant submanifold $S$ a \emph{Riemannian groupoid}. However, we don't know of any method to produce such metrics apart from the metrics associated with 2-metrics.
\end{remark}


\medskip

We can easily deduce from Theorem \ref{thm:main:2} the main results on linearization of proper groupoids that one can find in the literature \cite{cs,ppt,wein1,wein2,zung}:

\begin{corollary}[Linearization of proper groupoids]\label{linearization.proper}
If $G\toto M$ is a Hausdorff proper groupoid and $S\subset M$ is a saturated embedded manifold, then $G$ is linearizable around $S$.
\end{corollary}

\begin{proof}
By Theorem \ref{thm:main:1}, we can endow our groupoid with a 2-metric. By Theorem \ref{thm:main:2}, we obtain a groupoid neighborhood $\tilde V\toto V$ of $G_S\toto S$ in the local model which can be embedded into $G\toto M$. The proof is completed by observing that in a proper groupoid every groupoid neighborhood contains a full groupoid neighborhood (cf.~Lemma \ref{lem:full:neighborhood}).
\end{proof}

\begin{corollary}[Invariant linearization of $s$-proper groupoids]
If $G\toto M$ is a Hausdorff groupoid whose source map is proper and $S\subset M$ is a saturated embedded manifold, then $G$ is invariantly linearizable around $S$.
\end{corollary}

\begin{proof}
Every $s$-proper groupoid is proper, and its orbits are stable, namely every neighborhood $U$ of a saturated embedded manifold $S$, contains a saturated neighborhood of $S$ (cf. \cite[Prop~5.3.3]{survey}). The proof now is clear.
\end{proof}

\begin{remark}
Invariant linearization of $s$-proper groupoids covers a large number of related classical results, on fibrations, group actions and foliations. But as explained in \cite{cs}, it does not imply the Tube Theorem for proper actions (cf. \cite[Thm. 2.4.1]{dk}), where invariant linearization holds without requiring s-properness. Theorem \ref{thm:main:2} also does not yield the Tube Theorem, but our strategy of proof should still work: if we use the metric described in Example \ref{ex:proper-actions}, then the admissible neighborhoods can be taken to be invariant. 
Nevertheless, a statement for Lie groupoids generalizing the Tube Theorem is still lacking. The one conjectured in \cite{survey} is that invariant linearization holds for any $s$-locally trivial proper groupoid.
\end{remark}

\medskip

We can also formulate an infinitesimal version of the linearization theorem, which gives a criterion to conclude that a given Lie algebroid does not admit a proper integration.

Given $G\toto M$ a Lie groupoid and $S\subset M$ a saturated embedded submanifold, we can define the {\bf infinitesimal local linear model} as the Lie algebroid of the local linear model, say $\Lie(\nu(G_S)\toto S)=A_{\nu(G_S)}\to S$. We will say that the groupoid is {\bf infinitesimally linearizable} around $S$ if there are opens $S\subset U\subset M$ and $S\subset V\subset \nu(S)$ and a Lie algebroid isomorphism
$$A_G|_U\cong A_{\nu(G_S)|_V}$$

\begin{corollary}
A Riemannian groupoid is infinitesimally linearizable around any saturated submanifold $S$.
\end{corollary}

\begin{proof}
Notice that at the infinitesimal level weak linearization and linearization agree. More precisely, given $\tilde U\toto U$ an open subgroupoid of $G\toto M$, the inclusion
$$(\tilde U \toto U)\to (G_U\toto U)$$
defines an isomorphism between the corresponding Lie algebroids. The result now follows from Theorem \ref{thm:main:2}.
\end{proof}

\begin{example}\label{ex:lie.alg.bundle}
Let $G\toto M$ be a Lie group bundle. In this case the linear local model around a point $x\in M$ can be identified with the product $G_x\times T_xM\toto T_xM$. Hence, if there exists a 2-metric in $G\toto M$ then the underlying bundle of Lie algebras $A_G\to M$ must be locally trivial.
\end{example}


\appendix

\numberwithin{equation}{section}
\section{Some technical background}  
\label{appendix:averaging}    

In this appendix, we recall the concept of quasi-action, with focus on the tangent lift of an action. We then introduce averaging operators for quasi-actions. This is a crucial technique that we use in the paper to construct 2-metrics on proper groupoids.


\subsection{The tangent lift of an action}


Unlike the group case, an action of a Lie groupoid on a manifold does not induce a tangent action on the tangent bundle. We do have natural actions in the normal directions of the underline foliation, the so-called normal representations that we have already discussed. Still, sometimes it is necessary to put them all in a common framework. This can be done with the help of a connection on the groupoid, which allow us to define a quasi-action on the tangent bundle.

\medskip


Let $G\toto M$ be a Lie groupoid, and let $E$ be a manifold. A {\bf quasi-action} $\theta:G\tilde\action E$ with {\bf moment map} $q:E\to M$ consists of a smooth map 
$$\theta:G\times_M E\to E \qquad (g,e)\mapsto \theta_g(e)$$ 
satisfying $q(\theta_g(e))=t(g)$ for all $(g,e)\in G\times_ME=\{(g,e)\in G\times E: s(g)=q(e)\}$.
In other words,  a quasi-action associates to each arrow \smash{$y\xfrom g x$} in $G$ a smooth map $\theta_g:E_x\to E_y$. The quasi-action is called:
\begin{enumerate}[(i)]
\item {\bf unital} if $\theta_{1_x}=\id_{E_x}$ for all $x\in M$;
\item {\bf flat} if $\theta_{g_1}\theta_{g_2}=\theta_{g_1g_2}$ for all $g_1,g_2\in \G2$.
\item {\bf linear} if $q:E\to M$ is a vector bundle and $\theta_g:E_x\to E_y$ is linear for all $g$.
\end{enumerate}
Thus, with these definitions, an \emph{action} is the same as a unital flat quasi-action and a \emph{representation} is the same as a linear action.


An action $\theta:G\action E$ can be lifted to a quasi-action of the action groupoid $G\ltimes E$ over the tangent bundle $TE$ with the help of a connection on the groupoid.
By a {\bf connection} $\sigma$ on the Lie groupoid $G\toto M$ we mean a vector bundle map $\sigma:s^*TM\to TG$ such that $\d s\cdot \sigma=\id_{s^*TM}$ and $\sigma|_M=\d u$. Hence, a connection yields a splitting for the following sequence of vector bundles over $G$:
$$\xymatrix{\ 0\ \ar[r] &\ t^* A \ \ar[r] &\  TG \ \ar[r]^{s_*} & \ s^*TM \ \ar@/^/[l]^{\sigma} \ar[r]& \ 0\ }.$$
A connection $\sigma$ is {\bf multiplicative} if its image is a subgroupoid of $TG\toto TM$. Using a partition of the unity, one can show that every Lie groupoid admits a connection (see e.g. \cite{ac}), however a groupoid may not have a multiplicative connection. For instance, a multiplicative connection for the pair groupoid $M\times M\toto M$ is the same thing as a trivialization of the tangent bundle $TM\to M$ which, of course, does not exists in general.

\begin{definition}
Given $\theta:G\action E$ an action and $\sigma$ a connection on $G$, the {\bf tangent lift} of $\theta$ is the quasi-action $T_\sigma\theta:(G\ltimes E)\tilde\action TE$ which has moment map the projection $TE\to E$ and is defined by
$$T_\sigma\theta:G\times_ M TE \to TE \qquad
(T_\sigma\theta)_{(g,e)}(v)=\d \theta(\sigma_g(\d_e q(v)),v).$$
By transposition, we define the {\bf cotangent lift} $T_\sigma\theta:(G\ltimes E)\tilde\action T^*E$:
$$T_\sigma^*\theta:G\times_M T^*E\to T^*E \qquad
\<(T_\sigma^*\theta)_{g,e}(\alpha),v\> = \< \alpha, (T_\sigma \theta)_{(g^{-1},ge)}(v) \>.$$
\end{definition}

We will often denote $(T_\sigma\theta)_{(g,e)}(v)$ just by $gv$ and similar for the cotangent lift. With these notations we have $\<g\alpha,v\>=\<\alpha,g^{-1}v\>$.
The tangent lift $T_\sigma\theta$ and the cotangent lift $T_\sigma^*\theta$ are both unital, but rarely flat. In fact, the tangent and cotangent lift are flat if and only if the connection is multiplicative. As we saw above in the example of the pair groupoid, this may be a quite restrictive condition, and that is why we need to consider quasi-actions to work with general groupoids.

\begin{example}
When $G\toto M$ is an \'etale groupoid the map $s_*$ is an isomorphism and there exists a unique connection, namely $\sigma=s_*^{-1}$. Moreover, this connection is multiplicative. Therefore, when working with \'etale groupoids (and orbifolds) the tangent and cotangent lift are canonically defined, and they are actual actions, which greatly simplifies the whole theory.
\end{example}

\medskip

Although the tangent and cotangent lift depend on the choice of a connection, their action along the directions transversal to the orbits is intrinsic:

\begin{proposition}\label{tan.lift.nor.representation}
Let $G\toto M$ be a Lie groupoid, $\sigma$ a connection, $\theta:G\action E$ an action and $O\subset E$ an orbit. Then $TO^\circ\subset T^*E$ is invariant for the cotangent quasi-action $T^*_\sigma\theta$, and the restriction $(T^*_\sigma\theta)|_{TO^\circ}$ agrees with the conormal representation of the action groupoid. Hence, it is an action which does not depend on $\sigma$.
\end{proposition}
\begin{proof}
The connection $\sigma$ consists in choosing for each $(g,e)\in G\times_M E$ a retraction for the linear map $d_{(g,e)}s^*:T_e^*E\to T_{(g,e)}^*(G\times_M E)$ in a smooth way. The value of such a retraction over the image $d_{(g,e)}s^*(T_e^*E)$ is totally settled. Writing $\tilde O=(G\times_M E)_O$, the result follows by noting that $T\tilde O^\circ\subset d_{(g,e)}s^*(T_e^*E)$.
\end{proof}

\medskip

We end this subsection by stating the following naturality properties of the tagent lift, whose proof are straightforward.

\begin{proposition}\label{naturality.tangent.lift}
Let $G\toto M$ be a Lie groupoid and fix a connection $\sigma$ on it. Then:
\begin{enumerate}[(i)]
\item If $\theta^E:G\action E$ and $\theta^F:G\action F$ are two groupoid actions with moment maps $q^E, q^F$, respectively, then for any equivariant map $p:E\to F$ the differential $\d p:TE\to TF$ is also equivariant for the tangent lifts $T_\sigma\theta^E$ and $T_\sigma\theta^F$;
\item If $\theta^1:G\action E$ and $\theta^2:G\action E$ are two commuting actions with moment maps $q_1,q_2:E\to M$, then the tangent lifts $T_\sigma\theta^1$ and $T_\sigma\theta^2$ also commute.
\end{enumerate}
\end{proposition}

%




\subsection{Haar systems and averaging methods}


Haar systems on Lie groupoids generalize Haar systems on Lie groups, they always exist for proper groupoids, and allow some averaging arguments on functions and sections of equivariant vector bundles. We show that this can even be extended so as to include vector bundles endowed with quasi-actions, and apply in this way averaging arguments to metrics.

\medskip



Recall that a \emph{smooth density} on a vector bundle $E\to M$ of rank $r$ is a nowhere vanishing smooth section $\mu$ of the trivial line bundle $(\wedge^r E)\otimes (\wedge^r E)$. For instance, when $E$ is orientable, any volume form in $E$, i.e., a nowhere vanishing section $\omega$ of $\wedge^r E$, determines a density $\omega\otimes\omega$. 

Let $G\toto M$ be a Lie groupoid with associated algebroid $A\to M$. Given a smooth density $\mu$ on the underlying vector bundle, we denote by $\mu^x$ the pullback density on $G(-,x)=s^{-1}(x)$ through the target map:
$$\begin{matrix}
\xymatrix{TG(-,x) \ar[d] \ar[r]^(.6){\phi} & A \ar[d] \\ G(-,x) \ar[r]^(.6)t & M}
  \end{matrix}
\qquad \phi(g,v)=(t(g), \d R_{g^{-1}}(v)).
$$
The family of densities $\{\mu^x\}_{x\in M}$ satisfies the following two properties:
\begin{enumerate}[(i)]
 \item (Smoothness) The function 
 $$ x\mapsto \int_{G(-,x)}f(g)\mu^x(g)$$ 
 is smooth for all $f\in C^\infty(G)$.
 \item (Right-invariance) For any arrow \smash{$y\xfrom h x$} and  $f\in C^\infty(G(-,x))$ one has:
 $$\int_{G(-,y)} f(gh)\mu^y(g)= \int_{G(-,x)}f(g)\mu^x(g).$$
In other words, we have $\mu^y = R_h^*(\mu^x)$, where $R_h:G(-,y)\to G(-,x)$ denotes right multiplication. 
\end{enumerate}

\begin{definition}
We say that $\mu$ is a {\bf normalized Haar density} if the family $\{\mu^x\}_{x\in M}$ also satisfies the following property:
\begin{enumerate}[(i)]
\setcounter{enumi}{2}
\item (Normalization) The support ${\rm supp}(\mu^x)$ is compact and  for all $x\in M$:
$$\int_{G(-,x)}\mu^x(g)=1.$$
\end{enumerate}
 \end{definition}


We have the following important fact:

\begin{proposition}
A proper groupoid $G\toto M$ admits a normalized Haar density. 
\end{proposition}

For a proof we refer to \cite{crainic,tu}. The basic idea is that one can construct such a density $\mu$ as the product $c\tilde\mu$ of a nowhere vanishing density $\tilde\mu$ and a cut-off function $c$. Here by a {\bf cut-off} function $c:M\to\R$ we mean a function whose support intersects the saturation of any compact set in a compact set, or equivalently, such that $s:\supp(c\circ t)\to\R$ is proper, plus the normalization condition:
$$\int_{G(-,x)}c(t(g))\mu^x(g)=1.$$


Now let $G\toto M$ be a Lie groupoid and let $\theta:G\action E$ be an action with moment map $q:E\to M$. We say that a function $f\in C^\infty(E)$ is {\bf $\theta$-invariant} if it is constant along the orbits, namely $f(\theta_g e)=f(e)$ for all $g, e$ for which the action is defined. A normalized Haar density allow us to construct for any $f\in C^\infty(E)=\Gamma(E,\R_E)$ a $\theta$-invariant function $I_\theta(f)$ by averaging over the orbits.

In the same fashion it is possible to average sections of more general vector bundles $\Gamma(E,V)$. More precisely, let $V\to E$ be a vector bundle, let $\theta^E:G\action E$ be an action and let $\theta^V:(G\ltimes E)\tilde\action V$ be a linear quasi-action. The main examples to keep in mind are the tangent and cotangent lifts of an action.
Writing $ge=\theta^E_g(e)$ and $gv=\theta^V_{(g,e)}(v)$, we say that a section $f\in\Gamma(E,V)$ is {\bf $\theta$-invariant} if  $f(g e)=g f(e)$ for all $g,e$ for which the action is defined.

\begin{definition}\label{def:aver.oper}
Given $G\toto M$ a Lie groupoid with normalized density $\mu$, $\theta^E:G\action E$ an action and $\theta^V:(G\ltimes E)\action V$ a linear quasi-action, the associated {\bf averaging operator} is defined by
$$I_\theta:\Gamma(E,V)\to \Gamma(E,V) \qquad 
I_\theta(f)(e):=\int_{G(-,x)} g^{-1}f(g(e))\mu^{x}(g) \qquad x=q(e).$$
\end{definition}


Note that $I_\theta(f)(e)$ only depends on the restriction of $f$ to the orbit of $e$.
The main properties of this averaging operator are summarized in the following proposition. The proof is straightforward.

\begin{proposition}\label{prop:aver.oper}
With the above notations, the following hold: 
\begin{enumerate}[(i)]
 \item If $\theta^V$ is flat then $I_\theta(f)$ is $\theta$-invariant for any $f$. 
 \item If $f$ is already $\theta$-invariant then $I_\theta(f)=f$.
 \item If $\theta^1,\theta^2:G\action E$ are two commuting actions then $I_{\theta^1}I_{\theta^2}(f)=I_{\theta^2}I_{\theta^1}(f)$.
 \item For any equivariant map $\phi:V_1\to V_2$ over vector bundles endowed with linear quasi-actions of $G\ltimes E$, the averaging operators commute with $\phi$.
\end{enumerate}
\end{proposition}









\end{document}